\documentclass[a4paper,12pt,reqno,openany]{amsart} %oneside, openany
\usepackage[left=2cm,right=2cm,top=2.5cm,bottom=2.5cm]{geometry}
\usepackage[english]{babel}
\usepackage[scr=rsfso]{mathalpha}

\usepackage{color} % où xcolor selon l'installation
\definecolor{Prune}{RGB}{99,0,60}
\usepackage{mdframed}
\usepackage{multirow} %% Pour mettre un texte sur plusieurs rangées
\usepackage{multicol} %% Pour mettre un texte sur plusieurs colonnes
\usepackage{scrextend} %Forcer la 4eme  de couverture en page pair
\usepackage{tikz}
\usepackage[absolute]{textpos} 
\usepackage{colortbl}
\usepackage{array}
\usepackage{geometry}
\usepackage{etoolbox}
\usepackage{graphicx}
\usepackage{caption}
\usepackage{booktabs}
\definecolor{myprpl}{RGB}{255,0,130}
\usepackage[colorlinks, citecolor=cyan, linkcolor=blue]
{hyperref}
\usepackage{caption}
\graphicspath{ {./images/} }
\usepackage{scrextend}
\usepackage{fancyhdr}
\usepackage{enumerate}

\usepackage{mathtools}
\usepackage{amsfonts}
\usepackage{amssymb}
\usepackage{amsmath}
\usepackage{amsthm}
\usepackage{float}
\usepackage{tikz-cd}

\usepackage{nicematrix}
\usepackage{caption}
\usepackage{array}
\newcolumntype{x}[1]{>{\centering\arraybackslash\hspace{0pt}}p{#1}}

\newlength{\minussignlength}
\def\JJ#1{^{\textcolor{gray}{\texttt{\#}} #1}}
\def\J{\JJ{j}}
\def\I{\JJ{1}}
\def\II{\JJ{2}}

\DeclareFontFamily{U}{wncy}{}
\DeclareFontShape{U}{wncy}{m}{n}{<->wncyr10}{}
\DeclareSymbolFont{mcy}{U}{wncy}{m}{n}
\DeclareMathSymbol{\Sh}{\mathord}{mcy}{"58} %<---- to get \Sh

\newtheoremstyle{myplain}
  {.5em}       % ABOVESPACE
  {.5em}       % BELOWSPACE
  {\itshape}   % BODYFONT
  {0pt}        % INDENT (empty value is the same as 0pt)
  {\scshape}   % HEADFONT
  {.}          % HEADPUNCT
  {5pt plus 1pt minus 1pt} % HEADSPACE
  {}           % CUSTOM-HEAD-SPEC

\newtheoremstyle{mydefinition}
  {.5em}       % ABOVESPACE
  {.5em}       % BELOWSPACE
  {\normalfont}% BODYFONT
  {0pt}        % INDENT (empty value is the same as 0pt)
  {\bfseries}  % HEADFONT
  {.}          % HEADPUNCT
  {5pt plus 1pt minus 1pt} % HEADSPACE
  {}           % CUSTOM-HEAD-SPEC

\newtheoremstyle{myremark}
  {.5em}       % ABOVESPACE
  {.5em}       % BELOWSPACE
  {\normalfont}% BODYFONT
  {0pt}        % INDENT (empty value is the same as 0pt)
  {\itshape}   % HEADFONT
  {.}          % HEADPUNCT
  {5pt plus 1pt minus 1pt} % HEADSPACE
  {}           % CUSTOM-HEAD-SPEC

\numberwithin{equation}{section}

\theoremstyle{myplain}% default
\newtheorem{thm}{Theorem}[section]
\newtheorem{lem}[thm]{Lemma}
\newtheorem{cor}[thm]{Corollary}
\newtheorem{prop}[thm]{Proposition}

\theoremstyle{mydefinition}
\newtheorem{defn}[thm]{Definition}
\newtheorem{exmp}[thm]{Example}

\theoremstyle{myremark}
\newtheorem{rem}[thm]{Remark}

\newtheorem*{dscr}{Description}
\newtheorem*{intprf}{Proof}
\newtheorem*{intprfsk}{Sketch of Proof}

\newenvironment{prf}[1][]
    {\ifstrempty{#1}{\begin{intprf}}
                    {\begin{intprf}[\textit{#1}]}}
    {\qed\end{intprf}}

\title
[Numerical Calculations of Periods on Schoen's Class of Calabi--Yau Threefolds]
{{\large{
\centerline{Numerical Calculation of Periods on}
\centerline{Schoen's Class of Calabi--Yau Threefolds}}}
\hspace{0pt}
\\\;\vspace{-4pt}
\\{$\;$\normalfont{Azur ĐonlagiĆ}%$\,^\star$
}}
\begin{document}

\begin{abstract}
\vspace{-7pt}
Through classical modularity conjectures, the period integrals of a holomorphic $3$-form on a rigid Calabi--Yau threefold are interesting from the perspective of number theory. Although the (approximate) values of these integrals would be very useful for studying such relations, they are difficult to calculate and generally not known outside of the rare cases in which we can express them exactly. %In particular, naive numerical calculations fail to converge quickly.
In this paper, we present an efficient numerical method to compute such periods on a wide class of Calabi--Yau threefolds constructed by small resolutions of fiber products of elliptic surfaces over $\mathbf P^1$, introduced by C. Schoen in his 1988 paper. Many example results are given, which can easily be calculated with arbitrary precision. We provide tables in which each result is written to a precision of 30 decimal places and then compared to integrals of the appropriate modular form, to confirm accuracy.
\end{abstract}
\maketitle%\footnote{footnote}

\vspace{-27pt}
\tableofcontents
\vspace{-37pt}

\section{Introduction and Notation}

Given a Calabi--Yau threefold $X$ over $\mathbf C$ with a fixed nowhere-zero holomorphic $3$-form $\omega$ (which is unique up to multiplication by a scalar, by the definition of $X$), we consider the \textit{periods} of $\omega$ on $X$. These are the complex numbers in the image of the map
$$\mathcal I : \mathrm{H}_3(X,\mathbf Z)\rightarrow\mathbf C$$
defined by integration of $\omega$ over $3$-cycles. Such values are interesting from the point of view of arithmetic (if $X$ is defined over $\mathbf Q$, they can be related to periods of certain modular forms;~for examples, see \cite{BKSZ}, \cite{Chm21}, \cite{CvS19}, \cite{Hof13}), however they are usually difficult to compute explicitly. In particular, naive numerical calculations fail to converge quickly. For some cases~in which it has been done, see \cite{LPV24} (especially \textsection1.2 there), \cite{Srt19} and \cite{RS20}.

In his paper \cite{Sch88}, C.\ Schoen constructed a ``large, yet quite tractable'' class of Calabi--Yau threefolds. The goal of this paper is to describe a numerical method for computing the group of periods on a given threefold in Schoen's class
and to give many explicit examples. %, which naturally splits into two steps, to be explained below (more details are given in the next section).
%For singular models of these threefolds, such a procedure has independently and simultaneously been outlined in \cite{PP25}. The advantage of our method is that it allows for treatment of both singular models and their ``small resolutions'' (the ultimate result of Schoen's construction; see below), as well as similarly-constructed threefolds which are not Calabi--Yau, and we provide many explicit examples. Our algorithm splits~into~two steps, which we now explain along with the structure of the paper:
For singular models of these threefolds, such a procedure has independently and simultaneously~been~outlined in \cite{PP25}. The main advantage of our method is that, in addition to the singular models,~it allows for treatment of their ``small resolutions'' (the ultimate result of Schoen's construction; see below), as well as similarly-constructed non-Calabi--Yau threefolds. Another advantage is our explicit construction of $3$-cycles which paints a clear geometric picture. The algorithm~splits into two steps, which we now explain along with the structure of the paper:
\medskip

Schoen's construction rests on small resolutions $\widehat{X}$ of a fiber product $X$ of two elliptic surfaces $E\I,E\II$ over $\mathbf P^1$ with semistable singular fibers. % over a projective curve (which we will always take to be $\mathbf P^1$ in our examples).
We review the details of this construction in Section 2 and select three ``types'' (Type I, II, III) of examples of such threefolds on which we illustrate our computation method. Our object of interest in later sections is thus in particular a fibration $X\rightarrow\mathbf P^1$ (with section) with generic fiber a product of two elliptic curves.

Let $\Sigma$ be the finite set of points in $\mathbf P^1$ with singular fibers in $X$. In Section 3, we explain the first step of the computation: to produce integrals of $\omega$ over $3$-chains in $X$ which (imprecisely, but intuitively) lie above a path in $\mathbf P^1\setminus\Sigma$. Given such a path $\gamma$, we may think of these $3$-chains simply as continuous families, indexed by $z\in\gamma$, of $2$-cycles in $X_z$. As we will see, we may even assume that each such $2$-cycle is a product of two loops (one in each of the two elliptic curves). We refer to these integrals as our \textit{partial periods} (see Definition~\ref{defnpartres}).

Avoiding the slow convergence of the naive approximation of these partial periods, we show in Example \ref{exmpmain} how to calculate them by solving differential equations (by implementing the ``Frobenius algorithm''; see Theorem \ref{frobsol}) coming from the Gauss--Manin connection. As this method is quite involved, in Section 4 we work out an explicit example (of type II).
%As this method is a bit involved, not to mention computationally heavy, we devote Section 4 to a completely explicit example (of Type II, as discussed above) of this calculation, which is carried out using the computer algebra program Maple.
\smallskip

The second step in our main computation is to combine the many partial periods calculated~in the preceding sections into periods of ($X$ and) $\widehat{X}$. The main difficulty lies in finding a ``recipe'' to exactly characterize these periods among all possible linear combinations of this raw data. In Section 5, such a characterization (Theorem \ref{thmperMV}) is found in terms of vanishing cycles of $X$ and $\widehat{X}$ at the points in $\Sigma$. In particular, it is shown that \textit{all} the periods of $X$ and $\widehat{X}$ can be expressed as combinations of our partial periods, which is not obvious (this claim is similar to what is proven in \cite[\textsection 3]{Sho81}, but much simpler).

Finally, in Section 6 we determine the vanishing cycles in all singular fibers (of $X$ and of $\widehat{X}$) in terms of the vanishing cycles on the elliptic surfaces $E\I, E\II$, which are easy to understand as their singular fibers are semistable by assumption. Moreover, we show how to translate~these statements into a very explicit form which is easily used (by a computer algebra program, such as Maple) to finish the main computation.

%More details on these main steps will be found in the next section, after we introduce the main objects of study.
Section 7 contains the results of this computation for all examples of Types I, II and III (more details on this at the end of the next section), based on explicit formulas from \cite{Her91} and \cite{Baa10}. More concretely, we give in each example a $\mathbf Z$-basis of the lattice $\mathrm{im}(\mathcal I)\subseteq\mathbf C$ of periods (of $X$ and of $\widehat{X}$).
All our symbolic and explicit calculations were done in Maple$^{_\mathrm{TM}}$ (see \cite{Maple}, v.\ 2022 and 2023) and the results are presented with 30 decimal digits of precision.
\medskip

A brief remark on our notation: Everywhere in the paper, we use the superscript $\J, j=1,2$ in all calculations which involve the elliptic surfaces $E\I,E\II$ and related maps and quantities specific to the two surfaces (such as their defining polynomials $g_2\J,g_3\J$, for example). This is made to distinguish them from the various other superscripts and subscripts used.

Unless stated otherwise, singular cohomology will be taken with coefficients in $\mathbf C$, and singular homology with coefficients in $\mathbf Z$. That is, we write: $\mathrm{H}_k(-) = \mathrm{H}_k(-,\mathbf Z)$ and $\mathrm{H}^k(-) = \mathrm{H}^k(-,\mathbf C)$.
All varieties considered in this paper are algebraic and over $\mathbf C$.
\medskip

The research which led to this paper was done as part of a project led by Sławomir Cynk at the Jagiellonian University in Cracow and supported by the National Science Centre, Poland, with grant No.\! 2020/39/B/ST1/03358.

The author is very thankful to prof.\! Cynk for this opportunity and continued guidance during the writing of this paper, as well as for always motivating the author to improve himself in both algebraic geometry and the Polish language. He would also like to thank the remaining staff at the UJ Institute of Mathematics for an excellent education and for their efficient administration, most of all prof.\! Dominik Kwietniak without whom his personal life would~have~now~been~much different. Finally, he would like to thank all his friends in Poland~for the many great memories~of his years spent there, academic or otherwise.

\section{Schoen's Construction and the Problem Statement}\label{schsec}

A few words about C. Schoen's construction of Calabi--Yau threefolds from \cite{Sch88}:

\begin{dscr}
Suppose we are given two elliptic surfaces $E\I, E\II$ over the base $\mathbf P^1$. Their fiber product $X\coloneqq E\I \!\times_{\mathbf P^1}\! E\II$ is a three-dimensional algebraic variety over $\mathbf C$. Directly blowing-up its singularities, we get a smooth variety, so a threefold $\widetilde{X}$. However, we generally~do~not~have much control over $\widetilde{X}$. For example, if $X$ has an ordinary double point (i.e. it's locally biholomorphic to the locus of $x_1x_2-x_3x_4=0$ around $0\in\mathbf C^4$), also called a ``node'', then this point gets replaced in $\widetilde{X}$ by an exceptional divisor, which is a smooth quadric.

Instead, ``small resolutions'' are recalled in \cite[\textsection 1]{Sch88}, which replace a node by a line. This is done by (locally-analytically) blowing-up along a surface $S$ passing through the singularity. Because $S$ is defined by an invertible sheaf away from the node, the effect of the blow-up is nontrivial only at the node, which gets replaced by a line. If the starting variety $X$ has only nodes as singularities, then applying this procedure several times gives a (smooth) threefold~$\widehat{X}$. Of course, $\widehat{X}$ is dependent on the surfaces $S$ chosen at each node (although, further blowing-up $\widehat{X}$ at the newly-created exceptional lines again gives $\widetilde{X}$ as above; thus the small resolutions lie in-between $X$ and its usual blow-up $\widetilde{X}$) and $\widehat{X}$ is generally not an algebraic variety!

The canonical sheaf $\Omega^3$ of $\widehat{X}$ agrees with the pull-back of the dualizing sheaf $\Omega^\circ$ of $X$ (in the sense of Serre duality, \cite[III.7]{Har77}) everywhere except at most on the exceptional lines which make up a set of codimension $2$, but then the two invertible sheaves are equal. This is the kind of control for which we needed small resolutions (without going into the Euler characteristics studied by Schoen in \cite[\textsection 5]{Sch88}), as we can now explicitly find a global section of $\Omega^3$.
Finally, we would like for the threefold $\widehat{X}$ to be a Calabi--Yau manifold and the discussion so far indicates three conditions that should be satisfied:

\begin{enumerate}[$\hspace{0.75 cm}\bullet$]
    \item The fiber product $X$ needs to have only nodes as singularities, so that we can construct~$\widehat{X}$.% the small resolutions.
    \item The dualizing sheaf $\Omega^\circ$ of $X$ needs to be trivial, for the same to hold for the canonical sheaf $\Omega^3$ of $\widehat{X}$.
    \item The small resolution $\widehat{X}$ needs to be a K\"ahler manifold.
\end{enumerate}

Let the elliptic surfaces $E\I,E\II$ be relatively minimal. It is proven in \cite[\textsection 2]{Sch88} that the second condition above is satisfied if and only if $E\I,E\II$ are both rational (with $0$-sections). Furthermore, it is shown that the first condition is satisfied if both surfaces have only singular fibers of type $I_n$ in the Kodaira classification, which are also often called ``semistable fibers''. Then a point $(w\I,w\II)\in X_z$ over some $z\in\mathbf P^1$ is a node if and only if $w\I$ and $w\II$ are both singular points in the (necessarily also singular) fibers $E\I_z$ and $E\II_z$. Then $z\in\Sigma\I\cap\Sigma\II$, where $\Sigma\J$ denotes the set of points with singular fibers in $E\J$.

In our situation, the third condition is equivalent to $\widehat{X}$ being a projective algebraic variety. Since an algebraic blow-up of a projective variety is a projective variety, we only need to check that the surfaces $S$ in the construction of small resolutions (with at least one passing~through each node) can be chosen so that they are all smooth closed algebraic subvarieties of codimension $1$ in $X$. 
Following this reasoning, it is proven in \cite[\textsection 3]{Sch88} that the third condition is satisfied if (and only if, up to isogeny in the first case below) one of the following two cases holds:

\begin{enumerate}[$\hspace{0.75 cm}\bullet$]
    \item $E\I = E\II$ %(it suffices that $\Sigma\I=\Sigma\II$ for Beauville surfaces; see Type I in Definition \ref{defntypes} below)
    \item $E\I\neq E\II$ and, for every $z\in\Sigma\I\cap\Sigma\II$, neither $E\I_z$ nor $E\II_z$ are of type $I_1$
\end{enumerate}

To conclude, we want that $E\I,E\II$ are two relatively minimal, rational elliptic surfaces over $\mathbf P^1$ which have only semistable singular fibers, and falling into one of the two cases listed here.
This explains the construction of a Calabi--Yau threefold $\widehat{X}$, which we will refer to as “Schoen’s construction” going forward. While $\widehat{X}$ itself depends on a choice of small~resolutions (through the choice of subvarieties $S$ in $X$), it turns out that its periods do not; see Remark \ref{remcuriosity}.
%The construction of a Calabi--Yau variety $\widehat{X}$ depends on a choice of subvarieties $S$ of $X$ (but the periods of $\widehat{X}$ do not; see Remark \ref{}). We refer to it as ``Schoen's construction'' going forward.\;\;
\end{dscr}

%Having recalled Schoen's construction, we want to systematically list some explicit examples to which we apply our computation method.
%While the threefold $\widehat{X}$ depends on the choice of subvarieties $S$ of $X$ in Schoen's construction, the periods of $\widehat{X}$ do not (see Remark \ref{}).

We now want to systematically list some explicit examples to which we apply our computation method.
For simplicity of calculation, we restrict ourselves to examples in which the singular fibers appear only over real points, that is $\Sigma\subseteq\mathbf P^1_{\mathbf R}\subseteq\mathbf P^1$. However, this assumption is in no way a necessary prerequisite to the method of computation described in this paper.
\begin{defn}\label{defntypes}
Let $E\J\rightarrow\mathbf C,\; j=1,2$, be rational elliptic surfaces with section and semistable fibers, all lying over $\Sigma\J\subseteq\mathbf R$. We construct $\widehat{X}$ as above and distinguish three \textit{types} of examples:
\begin{itemize}
    \item \textbf{Type I:} We take $E\I=E\II$ to be a (relatively minimal) elliptic surface with $\texttt{\#}\Sigma\J = 4$, i.e. a \textit{Beauville surface}. The list of all six such surfaces is given in \cite{Bea82}. Only four of them have all singular fibers lying above real points; they are determined by the Kodaira types of their singular fibers: $(I_2,I_2,I_4,I_4),\;(I_1,I_1,I_2,I_8),\;(I_1,I_1,I_5,I_5),\;(I_1,I_2,I_3,I_6)$. Explicit Weierstrass models of these surfaces can be found in \cite{Her91}.
\smallskip
    
    \item \textbf{Type II:} We take $E\I$ to be the surface with singular fibers $(I_1,I_2,I_3,I_6)$ from the item above. We then also take $E\II$ to be a copy of $E\I$ with a different map to $\mathbf P^1$, given by composing the map of $E\I$ with a M\"obius transformation. In particular, any nontrivial permutation of the triple of fibers of types $(I_2,I_3,I_6)$ can be uniquely attained in this way. Hence there are five examples of this type, first studied by M. Sch\"utt in \cite{Stt04}.
\end{itemize}
%As then exactly three singular fibers of each surface coincide with those of the other surface, and the remaining singular fibers are of type $I_1$, we again end up with a rigid Calabi--Yau manifold after following Schoen's construction.
Note that both Type I and Type II satisfy the conditions for Schoen's construction, hence $\smash{\widehat{X}}$ is a Calabi--Yau threefold. Moreover, the set $\Sigma\I\cap\Sigma\II$ has at least three elements in both cases. By \cite[Theorem 7.1]{Sch88}, this condition (together with the rest of our assumptions) ensures~that all the varieties $\widehat{X}$ of Type I and Type II are \textit{rigid} (have no infinitesimal deformations).~Since $\widehat{X}$ is Calabi--Yau, this is equivalent to saying that $\dim_{\mathbf C}\mathrm{H}^3(\widehat{X})=2$. %Then $\mathrm{rk}_{\mathbf Z}\,\mathrm{H}^3(X)=2$, so the $\mathbf Z$-module of periods of $\widehat{X}$ also has rank $2$.

For the third and final type, we relax the rationality condition on $E\I$ and $E\II$:
\begin{itemize}
    \item \textbf{Type III:} We take $E\I=E\II$ to be the modular surface $M$ over (the compactification of) the modular curve $X_1(N),N\geq 4$. The compactification of such a curve is isomorphic to $\mathbf P^1$ when $N\leq 10$ or $N = 12$, and the maps $M\rightarrow X_1(N)$ can be found in \cite{Baa10}. Parameterizing this surface over $X_1(N)\simeq \mathbf P^1$ is now easy to do in Maple.
    
    Note that for $N<7$ these surfaces coincide with some of those already considered in the previous two types, and that surfaces over $X_1(9)$ and $X_1(12)$ have singular fibers over non-real points. This leaves only $X_1(7)$, $X_1(8)$ and $X_1(10)$.
\end{itemize}
Without the rationality condition, the same construction as in examples of Type I and II yields in examples of Type III a threefold $\widehat{X}$ which is not Calabi--Yau. 
However, our interest in these examples stems from the fact that, in \cite{Del71}, Deligne used the self-fiber product of a modular surface given by the congruence subgroup $\Gamma_0(N)$ to construct certain Galois representations associated to modular forms. Under this construction, period integrals of the self-fiber product (i.e. results of our Type III) correspond to periods of the appropriate modular form (and hence also special values of its L-function).
\end{defn}

Let us now examine the explicit form in which these examples are given and see how Type~III differs from Types I and II. As explained for each of the three types, (birational models of) the elliptic surfaces $E\I,E\II$ are all known in Weierstrass form over $\mathbf P^1$. This means that we are given homogeneous forms $g_2\J,g_3\J$ such that the (usually singular) surface defined by $$(y\J)^2 = 4(x\J)^3-g_2\J x\J-g_3\J$$ is isomorphic to $E\J$ over $\mathbf P^1\setminus\Sigma\J$. Here we put $g_2\J\in\mathcal O_{\mathbf P^1}(4k)$ and $g_3\J\in\mathcal O_{\mathbf P^1}(6k)$, for $k\in\mathbf Z_{\geq1}$, to keep the surface well-defined up to admissible transformation (for all examples in Definition~\ref{defntypes}, the forms $g_2\J,g_3\J$ will be presented explicitly together with the results in Section 7). Now, this Weierstrass equation as written here requires two small clarifications:

\begin{exmp}\label{exmpdef3form}
Fix a chart $(z\mapsto[z:1]) : \mathbf A^1\hookrightarrow\mathbf P^1$ on which we identify $g_2\J,g_3\J$ with functions $g_2\J(-,1),g_3\J(-,1)$. First, the above equation defines a priori a surface in $\mathbf A^2\times\mathbf A^1$, but it clearly extends to a closed subvariety of $\mathbf P^2\times\mathbf A^1$, and we will always use this without mention, as is usually done with elliptic curves. Second, and more importantly, it glues with the analogous surface defined over the chart $(w\mapsto[-1:w]) : \mathbf A^1\hookrightarrow\mathbf P^1$. We will call these charts $U^z$ and $U^w$.
It follows that there is a model $X^z$ of the restriction $X|_{U^z}$ (of $X = E\I\times_{\mathbf P^1}E\II$ to the base $U^z$) which is isomorphic to it over $U^z\setminus(\Sigma\I\cup\Sigma\II)$. Explicitly, $X^z$ is the closure of the set
$$\left\{(x\I,y\I,x\II,y\II,z)\in\mathbf A^2\times\mathbf A^2\times\mathbf A^1\left|\;
\begin{aligned}
    (y\I)^2 = 4(x\I)^3-g_2\I(z,1)x\I-g_3\I(z,1)\\
    (y\II)^2 = 4(x\II)^3-g_2\II(z,1)x\II-g_3\II(z,1)
\end{aligned}
\right.\right\}$$
inside $\mathbf P^2{\times}\mathbf P^2{\times}\mathbf A^1$. Over every point $[z:1]\in U^z$, there is a nowhere-zero holomorphic $1$-form
$$\omega\J_z\coloneqq \frac{\mathrm dx\J}{y\J} = \frac{2\mathrm dy\J}{12(x\J)^2-g_2}$$
in the fiber $E\J_{[z:1]}$. These $1$-forms vary holomorphically between fibers, and we may thus consider the nowhere-zero $3$-form $\widetilde{\omega}\coloneqq \omega\I\wedge\omega\II\wedge\mathrm dz$ on the regular part of $X^z$ (that is, $X^z$ without finitely many nodes). 
Working locally, we find that it lifts to the regular part of $X|_{U^z}$. Our~goal in this example is to show that this form extends to a global holomorphic $3$-form on $\widehat{X}$.

We proceed the same way over $U^w$. If $[z:1]=[-1:w]$, then $zw = -1$ and we calculate:
%\begin{align*}
%g_2\J(z, 1) = (-z)^{4k}g_2\J(-1, -1/z) = w^{-4k}g_2\J(-1, w)\\
%g_3\J(z, 1) = (-z)^{6k}g_3\J(-1, -1/z) = w^{-6k}g_3\J(-1, w)
%\end{align*}
$$g_2\J(z, 1) = w^{-4k}g_2\J(-1, w), \;\;\;\;
g_3\J(z, 1) = w^{-6k}g_3\J(-1, w)$$
The transition map $X^z|_{U^z\cap U^w}\xrightarrow{\;\;\sim\;\;}X^w|_{U^z\cap U^w}$ is hence given by
$$(x\I,y\I,x\II,y\II,z)\longmapsto (w^{2k}x\I,w^{3k}y\I,w^{2k}x\II,w^{3k}y\II,w), \;\;\textrm{ where }w = -1/z$$
and thus, for the analogous definition of $\omega\J_w$ over $U^w$, we have
$$\omega\I_w\wedge\omega\II_w\wedge\mathrm dw = (-z)^k\omega\I_z\wedge (-z)^k\omega\II_z\wedge\mathrm dz/z^2 = z^{2k-2}(\omega\I_z\wedge\omega\II_z\wedge\mathrm dz)$$
on the regular part of $X|_{U^z\cap U^w}$. We thus get $2k-1$ independent holomorphic $3$-forms
$$z^n\widetilde{\omega} = z^n(\omega\I_z\wedge\omega\II_z\wedge\mathrm dz) = (-w)^{2k-2-n}(\omega\I_w\wedge\omega\II_w\wedge\mathrm dw),\;\;\textrm{ for } n\in\{0,1,\ldots,2k-2\}$$
that lift to the regular part of $X$, which is isomorphic to $\widehat{X}$ with finitely many lines removed.

Since $\widehat{X}$ is a smooth threefold and lines are in codimension $2$, these forms extend to $2k-1$ independent holomorphic $3$-forms on $\widehat{X}$.
When $\widehat{X}$ is a Calabi--Yau manifold, such a form~is~unique up to scalar (and nowhere-zero), which forces $2k-1 = 1$. In particular, $k = 1$ for examples of Type I and Type II. On the other hand, $k > 1$ in Type III: The modular surfaces over $X_1(7)$ and $X_1(8)$ give $k = 2$, while the one over $X_1(10)$ gives $k = 3$. 
%\textcolor{red}{Our interest in examples of Type III comes from the fact that Calabi--Yau threefolds can appear as quotients of the threefolds in these examples, thus knowledge of their periods might still prove useful.}
\end{exmp}

In the following sections, we will describe a procedure to numerically compute the periods of the form $\widetilde{\omega} = \omega\I\wedge\omega\II\wedge\mathrm dz$ (more generally, $z^n\widetilde{\omega}$) introduced in Example \ref{exmpdef3form}, which will work for any threefold produced by Schoen's construction. The final result is a finite-rank sub-$\mathbf Z$-module of $\mathbf C$, the image of $\mathrm{H}_3(\widehat{X})$ under the integration map $\smash{\mathcal I:[\gamma]\mapsto\int_\gamma\widetilde{\omega}}$.
When this threefold is rigid and Calabi--Yau, such as in examples of Types I and II, then $\mathrm{rk}_{\mathbf Z}\,\mathrm{H}_3(\widehat{X}) = \dim_{\mathbf C}\mathrm{H}^3(\widehat{X})=2$ (see Definition \ref{defntypes}). Otherwise, this rank can be larger.

\begin{rem}\label{remint3formvanish}
Suppose that a $3$-cycle $\gamma$ of $X$ (or $\widehat{X}$) is contained in a single fiber $X_z$, for $z\in\mathbf P^1$. %It follows from the definition of $\widetilde{\omega}$ that $\int_\gamma\widetilde{\omega}=0$ (due to the wedge product with $\mathrm dz$). By working locally-holomorphically, we see that this is true even when $z$ is a singular point.
Then $\int_\gamma\widetilde{\omega}=0$ since the restriction of a meromorphic $3$-form to a complex surface is trivial.

By assumption, the surfaces $E\I,E\II$ have sections, which we denote by $\textbf 0\I,\textbf 0\II$, respectively. If a $3$-cycle $\gamma$ of $X$ is contained inside one of the surfaces $\textbf 0\I{\times_{\mathbf P^1}}E\II$ or $E\I{\times_{\mathbf P^1}}\textbf 0\II$, then again $\int_\gamma\widetilde{\omega}=0$ %(again, due to the definition of $\widetilde{\omega}$ with the wedge product).
for the same reason.
\end{rem}

Finally, note that $\widetilde{\omega}$ is defined on $X$ away from finitely many nodes (contained in the singular fibers, on which integration of the form $\widetilde{\omega}$ is, in any case, trivial by the above remark).~Therefore, we may also consider periods of $\widetilde{\omega}$ on $X$ in place of $\smash{\widehat{X}}$, which generally yields a $\mathbf Z$-module with a larger rank. As this is the intermediate step in the construction of $\widehat{X}$ before taking small~resolutions,~the hope is that the additional periods obtained in this way may also have~an~arithmetic interpretation.

The next two sections focus on a feasible algorithm for integration of $\widetilde{\omega}$ (between singular fibers, our so-called \textit{partial periods}; see Definition \ref{defnpartres}). From the point of view of the period computation~method which is the ultimate goal of this article, the difference between $X$ and $\widehat{X}$ comes into play in Sections 5 and 6, where we study how their vanishing cycles differ at each singular~fiber. Then, Section 7 will feature results for both threefolds $X$ and $\widehat{X}$, in each example of Type I, II or III (and $2k-1$ forms in Type III).

\section{Picard--Fuchs Equations and the Partial Periods}

We first recall the following more general situation: A smooth algebraic fibration $E\rightarrow B$ (for example, the restriction of an elliptic surface to smooth fibers) is, by Ehresmann's theorem, a $\mathcal C^\infty$ fiber bundle. We have the topological bundles $\mathcal{H}_r(E,\mathbf Z)$ and $\mathcal{H}^r(E,\mathbf C)$ of singular (co)homology groups $\mathrm{H}_r(E_b)$ and $\mathrm{H}^r(E_b)$, respectively, for $b\in B$. More precisely, these are \textit{local systems} of some rank $d\geq 0$, i.e. their fibers are discrete free modules (over $\mathbf Z$ and $\mathbf C$, respectively)~of~rank~$d$. Now, the bundle $\mathscr H\coloneqq\mathcal{H}^r(E,\mathbf C)\otimes\mathcal O_B$ is equipped with a flat connection $\nabla : \mathscr H\rightarrow\mathscr H\otimes\Omega^1_B$ (the \textit{Gauss--Manin connection}; see \cite[\textsection 9.2.1]{Voi02}) such that the sections of $\mathcal{H}^r(E,\mathbf C)\subseteq\mathscr H$ are exactly the solutions of $\nabla P = 0$.

Let $\omega$ be a section of $\mathscr H$. If, over a chart $U\subseteq B$, we fix a basis $\gamma_1,\ldots,\gamma_d$ of fibers of $\mathcal H_r(E,\mathbf Z)$, we may consider a (fiber-wise) linearly spanning set of sections taking values in periods of $\omega$
$$P_j(b)\coloneqq \int_{\gamma_j(b)}\omega(b), \;\;\textrm{ for which }P_j^{(k)}(b) = \frac{\mathrm{d}^k}{\mathrm{d}b^k}\!\int_{\gamma_j(b)}\omega(b) = \int_{\gamma_j(b)}(\nabla_{\mathrm{d}/\mathrm{d}b})^{k}\omega(b)$$
(cf. \cite[\textsection 4]{Grf68}). Over a dense open set $B_0\subseteq B$, all sections taking values in periods of $\omega$ then locally satisfy differential equations of some minimal order $d_0$ (independent of choice of $\gamma_j$), by the fiber-wise linear dependence of $\omega,\nabla_{\mathrm{d}/\mathrm{d}b}\,\omega,\ldots, (\nabla_{\mathrm{d}/\mathrm{d}b})^d\omega$. In particular, $d_0\leq d$.

\begin{exmp}\label{exmpdefgm}
In the situation of Section 2, this fibration is $X\rightarrow\mathbf P^1\setminus\Sigma$ (for $\Sigma = \Sigma\I\cup\Sigma\II$). We first observe that $d = 6$: Indeed, $X_z = E\I_z\times E\II_z$ and the rank of the homology group
$$\mathrm{H}_2(X_z)\cong(\mathrm{H}_0(E\I_z)\otimes\mathrm{H}_2(E\II_z))\oplus(\mathrm{H}_2(E\I_z)\otimes\mathrm{H}_0(E\II_z))\oplus(\mathrm{H}_1(E\I_z)\otimes\mathrm{H}_1(E\II_z))$$
is $6$ (the first two summands have rank $1$ each, and the last rank $4$). Now, taking $\omega_z = \omega\I_z\wedge\omega\II_z$ shows that $d_0\leq 4$ since the integral is zero over the first two summands (cf. Remark \ref{remint3formvanish}). %In fact, $d_0 = 4$ unless $E\I = E\II$, in which case $d_0 = 3$ (if $\gamma_1,\gamma_2$ is a basis of $\mathrm{H}_1(E\I_z)$, then $\gamma_1\otimes\gamma_2$ and $\gamma_2\otimes\gamma_1$ map to the same value by symmetry).

Moreover, it is well-known that the connection $\nabla$ admits a completely algebraic definition (see \cite{KO68}). This means that, if we look at the induced differential equation over the algebraic chart $[-:1]$ on $\mathbf P^1$, its coefficients are going to be given by rational functions (or polynomials, by multiplying out the denominator). We will see an explicit example of this in Section 4.
\end{exmp}

We will exploit the existence of this differential equation, but before this, we must explain how to find it in practice (and then also how to extract the right solutions). First, consider~a~classical example of this phenomenon, the original Picard--Fuchs equation:

\begin{exmp}
Recall that an elliptic curve $C$ in Weierstrass form is nonsingular if and only if $0\neq\Delta(C)\coloneqq g_2^3(C)-27g_3^2(C)$ (this corresponds to the definition $4g_2^3(C)-27g_3^2(C)$ used if we alternatively take the Weierstrass form to be $x^3-g_2x-g_3$ without the leading coefficient $4$, but we will not follow that convention). Elliptic curves in Weierstrass form are classified up to isomorphism by the $J$-invariant $J(C)\coloneqq g_2^3(C)/\Delta(C)$; in particular invariant under admissible transformations $(x,y)\mapsto (u^2x, u^3y)$ for $u\in\mathbf C^\times$, for which $(g_2,g_3)\mapsto (u^{4}g_2, u^{6}g_3)$.

Every elliptic curve (in Weierstrass form) can be expressed in the form $y^2 = 4x^3-(x+1)g$ (so $g = g_2 = g_3$) by an appropriate admissible transformation. Setting $g(z) = 27z/(z-1)$ defines an elliptic surface over $U\coloneqq\mathbf C\setminus\{1\}$ such that $E_z : y^2 = 4x^3-(x+1)g(z)$ and $J(E_z) = z$. %Thus the fibration $E\rightarrow U$ is exactly $J$.

We consider the differential form $\omega = \mathrm{d}x/y$ on the fibers and denote by $f$ a section taking values in its \textit{periods}, which satisfies a differential equation by the discussion at the beginning of this section. It is exactly in this context (\cite[p.34]{FK890}; for a modern calculation, see \cite[Theorem 7.1]{Mil21}) that the historically original Picard--Fuchs equation was calculated to be:
$$\frac{\mathrm{d}^2f}{\mathrm{d}z^2} + \frac{1}{z}\frac{\mathrm{d}f}{\mathrm{d}z} + \frac{31z-4}{144z^2(1-z)^2}f = 0$$
We will be more interested in the behavior at infinity. Since $g$ extends to $\mathbf P^1$, we may pullback the same differential equation by $w\mapsto z = 1/w$ to the other standard chart of $\mathbf P^1$ to get:
$$w\frac{\mathrm{d}^2f}{\mathrm{d}w^2} + \frac{\mathrm{d}f}{\mathrm{d}w} + \frac{31-4w}{144(w-1)^2}f = 0$$
The point $w = 0$ is a singular point of the above differential equation; this agrees with intuition, as it corresponds to $J(E_z) = z = \infty$, and the fiber $E_\infty$ clearly must be singular. The function $g(z) = \tilde{g}(z:1)$ extends to $\tilde{g}(1:w) = 27/(w-1)$. Let's define the function
$$h(w)\coloneqq \left(\frac{27}{w-1}\right)^{\! -1/4}\cdot {_2}H_1\!\left(\frac{1}{12},\frac{5}{12};1;w\;\,\!\!\!\right)$$
where ${_2}H_1$ is the hypergeometric function. Up to a choice of branch of $4$th root (differing by a factor of a $4$th root of unity), $h$ is a uniquely defined holomorphic function and even admits a Taylor series convergent in a disk of radius $1$ (since both the $4$th root and the hypergeometric function, by definition, admit such series). This Taylor series can (by Maple) be calculated explicitly. Furthermore, $h$ is readily shown, by the basic properties of hypergeometric functions, to satisfy the above differential equation.

It is now not difficult to see that there exists some scalar $c\in\mathbf C^\times$ such that $c\cdot h$ takes values in periods of $\omega$ around the point $J = \infty$. For example, one may observe that $E$ has a semistable singular fiber at $J = \infty$, whose integral monodromy matrix (by the Kodaira classification, see \cite[VI.2.1]{Mir89}) fixes a unique basis vector in $\mathrm{H}_1(E_z)\simeq\mathbf Z^2$, which gives a section of periods that agrees with $h$ up to scalar.
\end{exmp}

We have therefore found a formula to essentially determine one period in each fiber sufficiently close to the singularity at $J=\infty$. However, we may observe that our solution is %in particular
of the form $c\cdot h = c\cdot g_2^{-1/4}\cdot {_2}H_1(1/12,5/12;1;1/J)$. This is not some coincidence provided by our particular choice of elliptic surface $E\rightarrow B$. Instead, note that for any elliptic curve $C$ in Weierstrass form, an admissible transformation acts with $(g_2,\mathrm{d}x/y)\mapsto (u^4g_2, u^{-1}\mathrm{d}x/y)$. Thus the expression $$g_2^{1/4}\;\mathrm{d}x/y$$ gives the ``same'' $1$-form regardless of the particular Weierstrass form considered on $C$, and we have just proven (provided that $J(C)\neq 0,1$, at which points $h$ is not defined) that a period of this form has value $c\cdot{_2}H_1(1/12,5/12;1;1/J)$, for the same $c\in\mathbf C^\times$ as above. We may gather all of this~in~the following statement:

\begin{prop}\label{explperiod}
Given an elliptic surface $E\rightarrow B$ in Weierstrass form, for every nonsingular fiber $E_b$ such that $J(E_b)\neq 0,1$, there is some primitive (that is, not a multiple of any other) element $[\gamma_b]\in\mathrm{H}_1(E_b,\mathbf Z)\simeq\mathbf Z^2$ such that the following identity holds%, up to a factor of a $4$th root of unity:
$$g_2^{1/4}(E_b)\int_{\gamma_b}\frac{\mathrm{d}x}{y} = c\cdot {_2}H_1\!\left(\frac{1}{12},\frac{5}{12};1;\frac{1}{J(E_b)}\right)$$
up to a choice of branch of the $4$th root (a factor of a $4$th root of unity). %The constant $c$ is equal to $2\pi/\sqrt{4}{12}$.
\end{prop}

The assignment $b\mapsto[\gamma_b]$ may locally be chosen as a section of $\mathcal H_1(E,\mathbf Z)$.
%The identity written above extends away from $J=\infty$ by analytic continuation.
The constant~$c$~can be calculated at any point and yields $2\pi/\sqrt[4]{12}$. The ambiguity of a $4$th root of unity can in principle be reduced to a $2$nd root of unity (which is acceptable, since both $c\cdot h$ and~$-c\cdot h$~take~values in periods of $\mathrm dx/y$) by considering $(g_3/g_2)^{1/2}$ instead of $g_2^{1/4}$ (and adding a factor of $(27J/(J-1))^{1/4}$ for~a~fixed branch); however, this introduces additional complications and we do not bother~with constants as the final results of this paper are naturally defined up to a scalar factor. %The indivisibility (that is) of the above period in every fiber is also important only if we're interested in constant factors.

The important consequence for us is that, given an elliptic surface $E\rightarrow\mathbf P^1$ and a point~$[z_0:1]$ with $\lim_{z\rightarrow z_0}J(E_{[z:1]})=\infty$ (it is known that then in fact $E$ necessarily has a semistable singular fiber at $[z_0:1]$), we can calculate the Taylor series at $z_0$ of one map $z\mapsto \int_{\gamma_z}(\mathrm dx/y)$ taking~values in periods of $E_{[z:1]}$, which may then be used to find the Picard--Fuchs equation over a fixed chart. We apply a similar procedure (see Example \ref{exmpmain}) to our main example coming from Section~2:\;\;\;
%This Taylor series can then be used to find the Picard--Fuchs equation of the surface over a fixed chart $[-:1]$ by solving for sufficiently long polynomial coefficients. Moreover, we may apply a similar procedure to our main example coming from Section 2:

\begin{defn}\label{defnpartres}
Let $E\J\rightarrow\mathbf P^1$, $j=1,2$, be elliptic surfaces (in fixed Weierstrass forms) having singular fibers over $\Sigma\J\subseteq\mathbf P^1$, all assumed to be semistable. Denote by $U\hookrightarrow\mathbf P^1$~the~chart $z\mapsto [z:1]$. We consider the fiber product $X = E\I\!\times_{\mathbf P^1}\!E\II$ and the holomorphic~$3$-form~$\widetilde{\omega}$~introduced in Example \ref{exmpdef3form}, which takes the form $\omega\I\wedge\omega\II\wedge\mathrm{d}z$ over $U\setminus\Sigma$, for $\Sigma = \Sigma\I\cup\Sigma\II$.

Suppose that $a,b\in\Sigma$ and let $ab\subseteq U\setminus\Sigma$ denote an open path with ends $a$ and $b$. Consider a section $\ell$ of the local system $\mathcal H_2(X|_{U\setminus\Sigma},\mathbf Z)$, defined over $ab$. We will call values of the form
$$q^{ab}_\ell\coloneqq\int_{\ell\times ab}\widetilde{\omega}$$
the \textit{partial periods} of $\widetilde{\omega}$ (where $\ell\times ab$ is any $3$-cycle which follows the obvious definition~in~terms~of trivializations along $ab$).
The decomposition given by the K\"unneth formula (cf. Example~\ref{exmpdefgm})
$$\mathcal H_2(X|_{U\setminus\Sigma},\mathbf Z)\cong\mathcal H_2(E\I|_{U\setminus\Sigma},\mathbf Z)\oplus \mathcal H_2(E\II|_{U\setminus\Sigma},\mathbf Z)\oplus \left(\mathcal H_1(E\I|_{U\setminus\Sigma},\mathbf Z)\otimes\mathcal H_1(E\II|_{U\setminus\Sigma},\mathbf Z)\right)$$
together with Remark \ref{remint3formvanish} show that we may always assume $\ell$ to be a section of the rank-$4$ local system $\mathcal H_1(E\I|_{U\setminus\Sigma},\mathbf Z)\otimes\mathcal H_1(E\II|_{U\setminus\Sigma},\mathbf Z)$, otherwise the corresponding partial period equals $0$. In particular, when $\ell = \ell\I\otimes \ell\II$, we have:
$$q^{ab}_\ell = \int_a^b P_{\ell\I}(z) P_{\ell\II}(z) \;\!\mathrm{d}z,\;\;\textrm{ where }P_{\ell\J}(z)\coloneqq\int_{\ell\J(z)}\frac{\mathrm{d}x\J}{y\J}$$
The main goal of this section is to describe a feasible process for calculating such partial periods (as a naive approximation will generally not converge quickly enough for practical purposes). In practice, one wants to fix local sections $\ell_1\J,\ell_2\J$ of $\mathcal H_1(E\J|_{U\setminus\Sigma},\mathbf Z)$ which give bases in all fibers over $ab$. When $\ell= \ell_u\I\otimes \ell_v\II$, we will write $P_u\I = P_{\ell_u\I}$, $P_v\II = P_{\ell_v\II}$ and $q^{ab}_{u,v} = q^{ab}_\ell$.
\end{defn}

\begin{rem}\label{remrealroots}
If $E\I,E\II$ are given by real coefficients (of the homogeneous polynomials $g_2\J,g_3\J$) and if $ab\subseteq\mathbf R\subseteq U$, then there is a particularly nice choice of basis $\ell_1\J,\ell_2\J$. This will be the~case in all our examples of Type I, II or III, so we describe it here:

For $z\in ab$, we observe two cases for the distinct (there are no other singularities on this~path) roots $r_1\J(z),r_2\J(z),r_3\J(z)\in\mathbf C$ of the polynomial $4(x\J)^3-g_2\J(z)x\J-g_3\J(z)$:
\begin{enumerate}[$\hspace{0.75 cm}\bullet$]
    \item all three are real; then we can assume $r_1\J < r_2\J < r_3\J$ on $ab$
    \item one is real and two are complex-conjugate; then we can take $\mathrm{Im}(r_1\J) < \mathrm{Im}(r_2\J) < \mathrm{Im}(r_3\J)$ since if any two of these roots were to have equal imaginary parts, then they would themselves have to be equal; a contradiction 
    %no two such distinct roots can have equal imaginary parts at some point
\end{enumerate}
It thus makes sense to take the class $\ell_u\J(z)$ to be represented by a $2$-cycle lying ($2$ to $1$) above a straight line in $\mathbf C$ connecting $r_u\J(z)$ and $r_{u+1}\J(z)$. This defines the period functions $P_u\J$ up to sign (changing the branch of the square root in $\mathrm{d}x/y$ has the same effect as changing~the~orientation of the cycle). In the first case, one is real and the other imaginary, while in the second they are conjugate up to sign. All this is dependent on (and specific to!) the interval $ab$.
\end{rem}

\begin{exmp}\label{exmpmain}
Finally, we may explain the main idea of this part of the computation: There is, by Example \ref{exmpdefgm}, a differential equation $\Phi$ (with polynomial coefficients) of order $d_0\leq 4$ which is satisfied by every section $P$ over (a dense open set in) the chart $[-:1]$ taking values in periods of $\widetilde{\omega}$. To explicitly find $\Phi$, it suffices to take a Taylor series representing one such section $P$ and search for polynomials $p_j$ such that $\sum_{i=0}^4p_jP^{(j)}=0$. This is completely algorithmic, as it involves solving a system of linear equations for the coefficients of $p_j$ (where we assume $\deg p_j<N$ and this $N$ is increased until we get that such solutions exist).

The Taylor series that we use for this is of course that of $P = P\I\cdot P\II$ near~any~$s\in\Sigma\I\cap\Sigma\II$ (with both fibers semistable), where $P\J$ are given by Proposition \ref{explperiod} up to scalar.~Namely:
$$P(z)\coloneqq(g_2\I)^{-1/4}(z)\cdot (g_2\II)^{-1/4}(z)\cdot{_2}H_1\!\left(\frac{1}{12},\frac{5}{12};1;\frac{1}{J\I(z)}\right)\cdot {_2}H_1\!\left(\frac{1}{12},\frac{5}{12};1;\frac{1}{J\II(z)}\right)$$

For any point $z_0$, the equation $\Phi$ admits a $d_0$-dimensional complex vector space of solutions in a neighborhood $U$ of $z_0$, which are multivalued functions: In fact, we will show (Theorems \ref{holosol} and \ref{frobsol} below) that $d_0$ basis elements of the solution space can be found explicitly, as combinations of power series and powers of $\log$ (see Definition \ref{regsingdef}).
%in the form $z\mapsto(z-z_0)^\rho\sum_{i=0}^{d_0-}h_i(z)\log^i(z-z_0)$ (for $\rho\in\mathbf C$ and all $h_i$ holomorphic at $z_0$). 
Suppose this for~the~moment.~By just computing the values $P_j\I(z_i)P_k\II(z_i)$ for some $z_1,\ldots,z_{d_0}\in U$ and a choice of bases~$\ell_1\J,\ell_2\J$ (as in the preceding definition and remark), we are now able to interpolate the above solutions and determine an expansion around $z_0$ of the functions $P_j\I\cdot P_k\II$ for $j,k\in\{1,2\}$.
\end{exmp}

\begin{rem}\label{remwelldefde}
As we've noted, $d_0\leq 4$. However, when $E\I=E\II$, then $d_0\leq 3$ (because we~then have the identity $P_1\I\cdot P_2\II = P_2\I\cdot P_1\II$ by symmetry).

Our algorithm finds the minimal rational differential~equation satisfied by the explicitly~known function $P$, which is of some order $d'_0\leq d_0$. In practice, it often holds that $d'_0 = 3$~if~$E\I=E\II$ and $d'_0 = 4$ otherwise, thus $d'_0 = d_0$ and the equation found must then indeed be $\Phi$. This is in particular always true for our Types I and III ($d_0 = 3$) and Type II ($d_0 = 4$).
\end{rem}

Continuing Example \ref{exmpmain}, suppose that $z_0 = a\in\Sigma$. We define a function
\begin{equation}\label{eqintdefq}
    Q^a_{j,k}(z)\coloneqq\int^z_aP_j\I(t) P_k\II(t)\,\mathrm dt
\end{equation}
in a small neighborhood of $a$ (where it can also be computed explicitly by integrating the expansion of the product $P_j\I\cdot P_k\II$ term-by-term). Our goal is now to extend $Q_{j,k}^a$ along a fixed path $ab$ to $b\in\Sigma$, which gives us the partial period $q^{ab}_{j,k} = Q^a_{j,k}(b)$.

This is simple to do, as $Q_{j,k}^a$ satisfies the obvious differential equation $\Phi'$ (in which we raise the order of derivatives in $\Phi$ by one). Our expansion converges in a disk with radius bounded by the proximity of the nearest singularity of $\Phi'$. We may work with this to construct a sequence of disks (centered at singular or nonsingular points) covering the chosen path between $a$~and~$b$, and then analytically continue our solutions. In each consecutive disk, solve $\Phi$ and compare~the indefinite integrals of its solutions to the definite integral $Q_{j,k}^a$.

\begin{rem}\label{remuphpdeform}
Technical remarks: The comparison can be done on the values at any $d_0+1$ (the order of $\Phi'$) points in the extension, but better precision is obtained by a comparison of the first $d_0+1$ derivatives of $Q_{j,k}^a$ at a single point. Also, even if (as in Remark \ref{remrealroots}) $ab$ is contained in the real line, $\Phi'$ might still have a singularity in the interior of $ab$. We will avoid such singularities by always deforming $ab$ into the upper half-plane (see Figure \ref{analcontfig} in the next section).

This method is easily performed on a standard PC with a regular amount of computing power, and the results calculated in this paper have used up to $350$ coefficients of the considered series with $1000$ digits of working precision. In comparison, approximating $Q_{j,k}^a(b)$ by calculating the integral \eqref{eqintdefq} by quadrature would involve calculating elliptic integrals above many integration points $t\in ab$ (as opposed to only $d_0$ points $z_1,\ldots,z_{d_0}$, following Example \ref{exmpmain}), which can take more than half a minute per integration point at the required precision.
%In comparison, explicit calculation of even just one elliptic integral to calculate a single period, such as we need to do only $d_0$ times at the beginning, can take more than a minute.
%; thus a naive method of integrating period functions to get results of any meaningful precision would certainly not be feasible even on a much stronger device.
\end{rem} 

All that remains is to show that we can indeed find explicit expansions of $P_j\I\cdot P_k\II$ as stated above. To this end, we note that the equation $\Phi$ is Fuchsian (see the following definition)~by \cite[4.3]{Grf70}, and apply to it the general theory recalled below, with which we finish this section: 

\begin{defn}\label{regsingdef}
Consider an ordinary differential equation of the form
$$0 = f^{(n)}+\sum_{j=1}^{n}a_jf^{(n-j)}$$
and suppose that its coefficients are meromorphic on an open domain $U\subseteq\mathbf C$. A point $z_0\in U$ is a \textit{regular point} of the equation if $a_j$ is holomorphic at $x$ for all $j=1,\ldots,n-1$, otherwise it is a \textit{singular point}. However, if a singular point $z_0$ is such that $z\mapsto (z-z_0)^{j}a_j(z)$ is holomorphic at $z_0$ for all $j$, we call it a \textit{regular singular point} of the equation (otherwise it is an irregular singular point).

If a solution of the equation is given in a neighborhood of a singular point $z_0$ by the sum
$$z\mapsto (z-z_0)^\rho\sum_{k=0}^{n-1} h_k(z)\log^k(z-z_0), \;\;\;\;\;\;\; \rho\in\mathbf C\textrm{ and all }h_k\textrm{ are holomorphic at }z_0,$$
then we say it is a \textit{regular solution} at $z_0$. The branch of logarithm taken has no effect on the definition, however we will always take $\log(z) = \log|z|+i\cdot\mathrm{arg}(z)$ for $\mathrm{arg}(z)\in(-\pi,\pi]$ to agree with our calculations done in Maple. Similarly, we take any fixed branch of $z\mapsto z^\rho$.

It is a fact that a singular point is regular singular if and only if all the solutions of the equation are regular solutions at that point (see \cite[15.3]{Inc27}). If all the singular points of~the~differential equation are regular singular points, then the equation is said to be \textit{Fuchsian}.
\end{defn}

The following two statements are well-known and their proofs translate directly into efficient algorithms for finding holomorphic (resp.\ regular) solutions to Fuchsian equations. In the case of holomorphic solutions, this algorithm is a simple inductive construction of the required power series, while in the second case it is the classical \textit{Frobenius method}.

\begin{thm}[{\cite[\textsection 12.22]{Inc27}}]\label{holosol}
Let $\Phi$ be an ordinary differential equation as in Definition~\ref{regsingdef}. Any regular point $z_0$ of $\Phi$ has a neighborhood $V$ which admits a full $n$-dimensional $\mathbf C$-vector space of holomorphic solutions. This space has a basis $(f_j)$ such that $f_j^{(j)}(z_0) = 1$ and $f_j^{(k)}(z_0) = 0$ for $k<j$; thus the elements $f$ in the solution space of $\Phi$ around $z_0$ correspond bijectively to arbitrary $n$-tuples $(f(z_0),f'(z_0),\ldots,f^{(n-1)}(z_0))\in\mathbf C^n$.
\end{thm}
%\begin{prf}
%For any such $n$-tuple, we may input its values into $\Phi$ to get $f^{(n)}(z_0)$. Differentiating $\Phi$, we get a formula for $f^{(n+1)}(z_0)$ and we repeat inductively until we can express the entire Taylor series of a formal solution $f$. It is proven in \cite[12.22]{Inc27}, that this series converges in some neighborhood of $z_0$.
%In particular, we can take $n$ series constituting a basis as described above. The intersection $V$ of their disks of convergence is an open set sufficiently small for the convergence of any solution. Moreover, we know that these are all the solutions, by Corollary \ref{solspacedim}.
%It is known that the solution space is at most $n$-dimensional.
%
%Alternatively, we may derive this result as a special case of the next theorem and observe that we get indices $0,1,\ldots,n-1$. It 
%\end{prf}

%The proof of the following theorem provides an algorithm to compute expansions of regular solutions of $\Phi$, to be used repeatedly in our computation of partial periods. We will call it the Frobenius method, and we now sketch the proof given in \cite[\textsection 16]{Inc27}.

\begin{thm}[{\cite[\textsection 16.1]{Inc27}}]\label{frobsol}
Let $\Phi$ be an ordinary differential equation as in Definition~\ref{regsingdef}. Any regular singular point $z_0$ of $\Phi$ has a neighborhood $V$ admitting a full $n$-dimensional $\mathbf C$-vector space of regular solutions, in the sense of the definition given above.
\end{thm}

%\bigskip
%\medskip
\section{An Illustrative Example of the Computation}\label{exmpsec}

To explain some of the finer details that occur in this calculation, we will explicitly consider an example of Type II (see Definition \ref{defntypes}). Let $E\I$ be the elliptic surface with singular fibers of type $(I_1,I_2,I_3,I_6)$ over $\mathbf P^1$. By \cite{Her91}, the equation $y^2 = 4x^3-g_2\I x-g_3\I$ gives a Weierstrass model of this surface, where we put:
\begin{align*}
    g_2\I(X,Y)&\coloneqq 12(X^4 - 4X^3Y + 2XY^3 + Y^4)\\
    g_3\I(X,Y)&\coloneqq 4(2X^6 - 12X^5Y + 12X^4Y^2 + 14X^3Y^3 + 3X^2Y^4 + 6XY^5 + 2Y^6)
\end{align*}
Clearly, $g_2\I$ and $g_3\I$ are global sections of $\mathcal O_{\mathbf P^1}(4)$ and $\mathcal O_{\mathbf P^1}(6)$, respectively. Thus the $J$-invariant, given on each fiber by the following formula
$$J\I(X\!:\!Y)\coloneqq \left(\frac{(g_2\I)^3}{(g_2\I)^3-27(g_3\I)^2}\right)\!(X\!:\!Y) = \frac{4(X^4 - 4X^3Y + 2XY^3 + Y^4)^3}{27X^3Y^6(X - 4Y)(Y + 2X)^2} = \frac{(g_2\I)^3(X,Y)}{\Delta(X,Y)}$$
is a (well-defined, i.e. independent of rescaling of $[X\!:\!Y]$) meromorphic function on $\mathbf P^1$. Looking at the denominator $\Delta(X,Y)$ confirms that $E\I$ indeed has singular fibers of type $I_1,I_2,I_3,I_6$ over points $[4\!:\!1],[-1/2\!:\!1],[0\!:\!1],[1\!:\!0]$, respectively. We will write this set as $$\Sigma\I\coloneqq \{4,-1/2,0,\infty\}$$ in accordance with the chart $z\mapsto[z:1]$ (with image $\{Y\neq 0\}$).

Let $E\II$ be the pullback of $E\I$ by the automorphism $[X\!:\!Y]\mapsto [-Y/2-X\!:\!Y]$ of $\mathbf P^1$. This has the effect of permuting the points $-1/2,0,\infty$ and taking $4$ to $-9/2$ (recall that this follows Schoen's construction). Then $$\Sigma\coloneqq \Sigma\I\cup\Sigma\II = \{0,4,-9/2,-1/2,\infty\}$$ and we name its points:
$$
a\coloneqq -\frac{9}{2},\;\;\;\;\;\;
b\coloneqq -\frac{1}{2},\;\;\;\;\;\;
c\coloneqq 0,\;\;\;\;\;\;
d\coloneqq 4,\;\;\;\;\;\;
e\coloneqq \infty,
$$
The pairs of fibers of $E\I,E\II$ lying above the points $a,b,c,d,e$ are of Kodaira types, in order: $(I_0,I_1),(I_2,I_3),(I_3,I_2),(I_1,I_0),(I_6,I_6)$ (where $I_0$ means that the fiber is nonsingular)

%We distinguish calculations in the affine lines $\{Y\neq 0\}$~and $\{X\neq 0\}$ covering $\mathbf{P}^1$.
Recall the notation of Example \ref{exmpmain}. We will first calculate $q^{ab}_{j,k}$ along the real interval $ab$, then $q^{bc}_{j,k}$ along $bc$, and so on. As discussed in Remark \ref{remrealroots}, we choose an ordering of the three roots $(r_{1}\I,r_{2}\I,r_{3}\I)$ of the fibers of $E\I$ and the three roots $(r_{1}\II,r_{2}\II,r_{3}\II)$ of the fibers of $E\II$ separately above each interval $ab$, $bc$, etc. This is represented by the following figure:

\begin{figure}[H]
  \centering
  \captionsetup{width=\linewidth}
  \includegraphics[width=\linewidth]{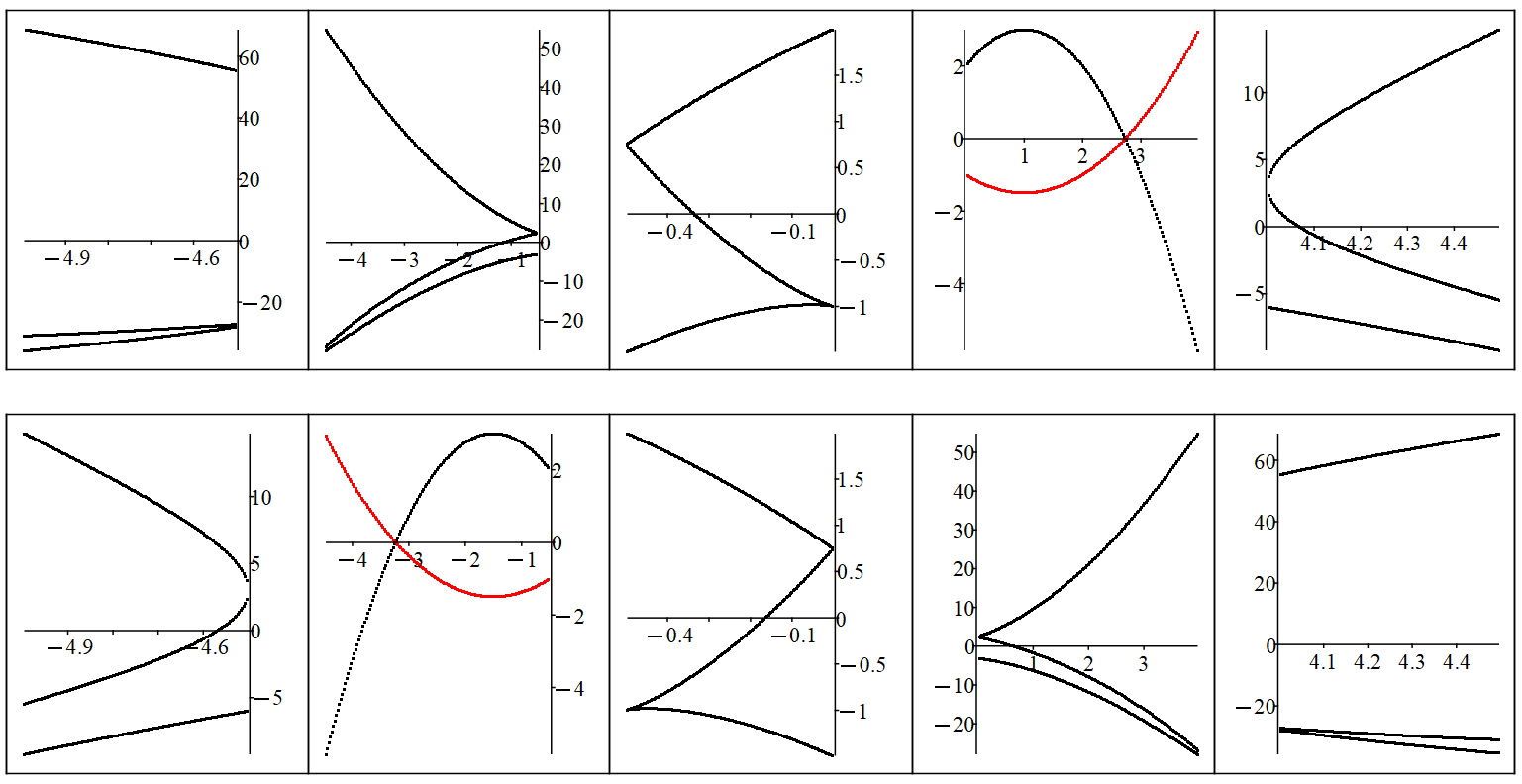}
  \caption{\textit{A schematic representation of the (real parts of) root triples $(r_{1}\J,r_{2}\J,r_{3}\J)$, $j = 1,2,$ of the surfaces $E\I$ (top) and $E\II$ (bottom) from the example above, on the intervals, in order, $(-5,a),(a,b),(b,c),(c,d),(d,5)$. Complex-conjugate pairs of roots are colored red.%, as opposed to the roots in the two top-left (resp.\ bottom-right) images which are merely too close to distinguish clearly.
  }}\label{figroots}
\end{figure}

Now, focus on the interval $ab$, on which Remark \ref{remrealroots} gives us a fiber-wise basis $\ell_1\J,\ell_2\J$ of $\mathcal H_1(E\J|_{U\setminus\Sigma},\mathbf Z)$. We need to explain how to explicitly work with the functions $P_j\I$, $P_k\II$, $Q^{ab}_{j,k}$ (in the notation of Example \ref{exmpmain}). Recall that $\omega\J = \mathrm dx\J/y\J$. We may fix a branch of~the~square root (choosing the other one has merely the effect of negating the classes $\ell_1\J,\ell_2\J$) and write these functions concretely as:
\begin{align*}\label{notation}
    P_1\J(z)&= 2\int_{r\J_{1}(z)}^{r\J_{2}(z)} \omega_+\J(z) &
    \omega_+\J(z)&\coloneqq\frac{\mathrm{d}x\J}{\sqrt{4(x\J)^3-x\J\cdot g_2\J(z,1)-g_3\J(z,1)}}\\
    P_2\J(z)&= 2\int_{r\J_{2}(z)}^{r\J_{3}(z)} \omega_+\J(z) &
    Q^a_{j,k}(z)&= \int^z_a P_j\I(t)\cdot P_k\II(t) \;\;\!\mathrm{d}t
\end{align*}
This allows us to directly compute $P_j\I(z_i)P_k\II(z_i)$ at a few points $z_i$ near $a$. We will also denote by $p\I$, $p\II$, $q^{ab}$ the vectors of these ($2$ or $4$, respectively) components. Any~arbitrary~choice~made here is inconsequential, as it at most multiplies the final results by a constant $\pm 1$.
%Our goal is to calculate the values $q_{j,k}^{ST}$, $j,k\in\{1,2\}$ for $ST\in\{AB,BC,CD,DE,EA\}$ and to see how they lift to form periods on the Calabi--Yau threefold after blowing up. Since~our~choice of $Y=1$ is arbitrary, the resulting integrals are defined only up to (a common) constant.

Next, we know an explicit solution of the differential equation $\Phi$ around $c = 0\in\Sigma\I\cap\Sigma\II$ (this can, in principle, be done around almost any point in the plane; we choose $c$ for convenience):
$$(g_2\I)^{-1/4}(z)\cdot (g_2\II)^{-1/4}(z)\cdot{_2}H_1\!\left(\frac{1}{12},\frac{5}{12};1;\frac{1}{J\I(z)}\right)\cdot {_2}H_1\!\left(\frac{1}{12},\frac{5}{12};1;\frac{1}{J\II(z)}\right) \;\; = 
\;\;\;\;\;\;\;\;\;\;\;$$
$$\;\;\;\;\;\;\;
=\;\; \frac{1}{3} - \frac{7}{18}z + \frac{367}{648}z^2 - \frac{5215}{5832}z^3 + \frac{416773}{279936}z^4 - \frac{12911183}{5038848}z^5 + \frac{613914581}{136048896}z^6 -\ldots
$$\vspace{0.1pt}

\noindent
After differentiating this power series $4$ times, we set up a system of infinitely many linear equations. Solving it for sufficiently many coefficients (in this case $11$) of each polynomial gives us the following differential equation, which must be equal to $\Phi$ (by Remark \ref{remwelldefde}):
\begin{align*}
     0 \;=\;& f^{(4)}\cdot z^2(4z + 1)(14z^2 + 7z + 360)(2z + 1)^2(2z + 9)^2(z - 4)^2\\
    +\;& f^{(3)}\cdot (13440z^{10} + 33600z^9 + 71360z^8 + \ldots + 21316608z^2 + 2332800z)\\
    +\;& f''\cdot (57344z^9 + 129024z^8 + 988352z^7 + \ldots + 103487688z^2 + 25093152z + 1866240)\\
    +\;& f'\cdot (75264z^8 + 150528z^7 + 2196896z^6 + \ldots + 47470612z^2 + 29102904z + 5298048)\\
    +\;& f\cdot (21504z^7 + 37632z^6 + 914880z^5 + \ldots - 7616436z^2 - 1921500z - 160704)
\end{align*}
(For completeness, note that we can find also the equations of the periods $P\I, P\II$ of $E\I, E\II$:
$P\I$ satisfies $z(2z + 1)(z - 4)f'' + (6z^2 - 14z - 4)f' + (2z - 2)f = 0$, degenerate at $b,c,d$, and
$P\II$ satisfies $z(2z + 9)(2z + 1)f'' + (12z^2 + 40z + 9)f' + (4z + 6)f = 0$, degenerate at $a,b,c$. Thus all these periods are holomorphic on $\mathbf C\setminus\{a,b,c,d\}$.)
\smallskip

We now start at point $a$ and apply the Frobenius method (Theorem \ref{frobsol}) to get $4$ linearly independent (multivalued) solutions of $\Phi$ around its regular singular point $a$. Next, we explicitly calculate the expansions around $a$ of all $4$ functions $P_j\I P_k\II$: by calculating first their values at some points $a+\varepsilon_1,a+\varepsilon_2,a+\varepsilon_3,a+\varepsilon_4$ in the interval $ab$ (in practice, $\varepsilon_i\ll 1$), then finding a unique linear combination of solutions which gives these values at these points.

The above equation is degenerate (i.e. has singularities) at $a,b,c,d$, but also at $-1/4$ and at $-1/4\pm (13/28)\sqrt{-119}$. The functions $P_j\I P_k\II$ will have only removable singularities at the extra points of degeneracy, being products of functions holomorphic at those points. Their actual singularities in $\mathbf{C}$ remain only at $a,b,c,d$. Nevertheless, convergence of series constructed by the Frobenius method is very slow at points of degeneracy and we will attempt to avoid~them (see Figure \ref{analcontfig} below).

Integrating term-by-term, we get an expansion of $Q^a_{j,k}(z)$ around $a$, for $j,k\in\{1,2\}$. Finally, we may extend the function along $ab$ as in the previous section to calculate $q^{ab}_{j,k} = Q^a_{j,k}(b)$, using the fact that $Q^a_{j,k}$ satisfies the equation $\Phi'$ of degree $5$ as follows. This procedure is sometimes called ``numerical holomorphic continuation'' (see \cite{Mez16} for a more general discussion, as well as information on its implementation in a dedicated SageMath package).

\begin{dscr}
Suppose that the vector space of solutions of $\Phi$ in some disk around a point $z_0$ is spanned by functions $F_1,F_2,F_3,F_4$, which we find as power series in $z-z_0$ (and $\log(z-z_0)$, if $z_0$ is a singular point). The solutions of $\Phi'$ are then locally spanned by the constant function $G_0 = 1$ and the primitives $G_i=\int \!F_i$, $\;i\in\{1,2,3,4\}$. (This method of calculation is much less computationally intensive than directly solving an equation of order $5$.)

To find coefficients $c_i\in\mathbf C$ such that $Q^a_{j,k}(z) = \sum^4_{i=0}c_iG_i$, we solve the following system of $5$ linear equations for some generic choice of point $z$ in the disk (the determinant of this matrix is not identically~$0$):
\vspace{-3pt}
$$
\left[\begin{matrix}
(Q^a_{j,k})(z)\\
(Q^a_{j,k})'(z)\\
(Q^a_{j,k})''(z)\\
(Q^a_{j,k})'''(z)\\
(Q^a_{j,k})''''(z)
\end{matrix}\right]
=
\left[\begin{matrix}
1 & G_1(z) & G_2(z) & G_3(z) & G_4(z)\\
0 & G_1'(z) & G_2'(z) & G_3'(z) & G_4'(z)\\
0 & G_1''(z) & G_2''(z) & G_3''(z) & G_4''(z)\\
0 & G_1'''(z) & G_2'''(z) & G_3'''(z) & G_4'''(z)\\
0 & G_1''''(z) & G_2''''(z) & G_3''''(z) & G_4''''(z)
\end{matrix}\right]
\cdot
\left[\begin{matrix}
c_0\\
c_1\\
c_2\\
c_3\\
c_4
\end{matrix}\right]
$$
Keep in mind that, if an expansion of $Q^a_{j,k}$ is known in some disk $D_1$ and we are extending it to a disk $D_2$ centered at $z_0$, then we should make sure that there is ample intersection between $D_1$ and $D_2$, and choose $z\in D_1\cap D_2$ not too close to any boundary (to avoid slow convergence).
\end{dscr}

Because $a = -9/2$ and $b = -1/2$, there are no singularities in the open interval $ab$. Thus, we may extend $Q^a_{j,k}$ along $ab$ by considering only $4$ series expansions in disks around the points, in order: $-9/2$ (of radius $4$), $-2$ (of radius $3/2$), $-1$ (of radius $1/2$), $-1/2$ (of radius $1/2$)

This completes the calculations in the interval $ab$; we may now proceed to the interval~$bc$.
%\bigskip

%Having done the interval $AB$, we already have found a basis $\{(Q^a_{j,k})'\}_{j,k\in\{1,2\}}$ of the solution space of the Picard--Fuchs equation around $B$. By setting up the roots in the interval $(B,C)$ as before, we can again define products $p_j\I p_k\II$ in this new interval. Calculating explicitly the values of these functions at $B+\varepsilon_1, B+\varepsilon_2, B+\varepsilon_3, B+\varepsilon_4$, we get a change of basis matrix $M_B$ relating the old (i.e. $(Q^a_{j,k})'$) and new definition of $p_j\I p_k\II$. For these new period functions, let:
%$$Q^B_{j,k}(z)\coloneqq \int_B^z p_j\I(Z) p_k\II(Z)\;\!\mathrm{d}Z, \;\;\;\;\textrm{ or equivalently, }\;\;\;\; Q^B\coloneqq M_B \cdot (Q^a-q^{AB})$$
%We can now repeat the earlier process to calculate $q^{BC}_{j,k} = Q^B_{j,k}(C)$: Recall that the integral in the definition of $q^{BC}$ is taken along the real segment connecting $B$ and $C$. Due to our choice of branch of logarithm (c.f. Definition \ref{regsingdef}), the considered branch of any multivalued solution around a point $S\in\mathbf R$ with singular fiber is continuous on the closed upper half-plane $\overline{\mathbf{H}_+}$. Thus we are free to deform the path when avoiding the extra points of degeneracy (see Figure \ref{analcontfig}) as long as we restrict our points of comparison to $\overline{\mathbf{H}_+}$.

\begin{figure}[H]
  \centering
  \captionsetup{width=\linewidth}
  \includegraphics[width=\linewidth]{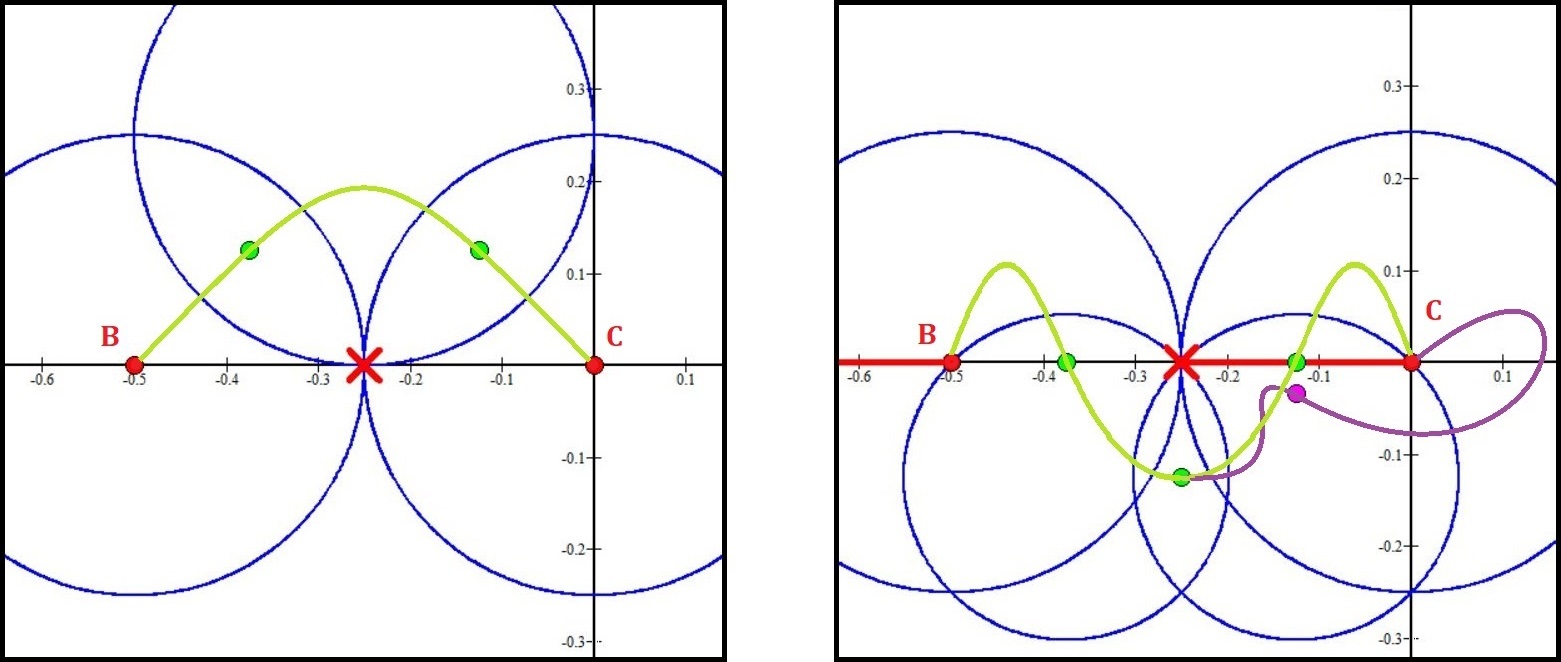}
  \caption{\textit{Two possible extensions from $b$ to $c$ (red dots). We start in a disk of convergence centered at $b$, and end in one at $c$. On the left picture, we can imagine we're integrating along the green curve as we pass through an intermediate disk (we extend $Q^b$ by comparing its $n$-th derivatives at the green points, $n\in\{0,1,2,3,4\}$). We avoid expanding a series at the point of degeneracy $-1/4$ (cross), as the convergence is slow there, hence the results get imprecise.\\
  On the right picture, we integrate through the lower half-plane.\! The result need not be different from the one obtained on the left; it remains the same if we compare at the green points. However, if we perform the last comparison at the purple point instead (with negative imaginary part), we pass under the line (red) marking the end of a branch of the logarithm function $z\mapsto\log(z-c)$ (we let the values on this line continue the branch on $\mathbf{H}_+$, which is consistent with the convention used in Maple) and the effect is as if winding around $c$ by integrating along the purple line, giving a different result. This is avoided by restricting ourselves to $\overline{\mathbf{H}_+}$.}}\label{analcontfig}
\end{figure}

Before explaining how the different intervals (and the transitions between them) are handled, it is convenient now to compare what happens if an interval contains a singular point of $\Phi'$,~such as $-1/4\in bc$. The figure presented above illustrates why we always choose to work in the closed upper half-plane $\overline{\mathbf{H}_+}$ (as already noted in Remark \ref{remuphpdeform}).
Up to accounting for singular points, the procedure used to produce $q^{ab} = (q^{ab}_{1,1}, q^{ab}_{1,2}, q^{ab}_{2,1}, q^{ab}_{2,2})$ can now simply be repeated to produce vectors $q^{bc}$, $q^{cd}$, $q^{de}$, $q^{ea}$. Note that, since $e = \infty$, the last two should be calculated in the other chart, $w\mapsto[-1:w]$ (cf. Example \ref{exmpdef3form}). This does not present any difficulty (and it is,~moreover, reassuring to check that $q^{ab}$ yields the same values when computed in both charts). 

\begin{rem}
Very importantly, the natural choice of bases $\ell_1\J,\ell_2\J$ is dependent on the interval (recall that this convention was introduced in Remark \ref{remrealroots}). We will say it is \textit{adapted} to it. Let $$\ell^{ab}\coloneqq (\ell_1\I,\ell_2\I)\otimes(\ell_1\II,\ell_2\II)  = (\ell_1\I\otimes\ell_1\II,\; \ell_1\I\otimes\ell_2\II,\; \ell_2\I\otimes\ell_1\II,\; \ell_2\I\otimes\ell_2\II)$$
denote the vector of bases adapted to $ab$ in our calculations. To have any chance of combining these partial periods into periods of $\widetilde{\omega}$, we would like to compare $\ell^{ab}$ and $\ell^{bc}$ over the~closed~upper half-plane $\overline{\mathbf{H}_+}$. Calculating this at any point yields the \textit{transformation matrix} $M_b$, a $4\times 4$ integer matrix such that $\ell^{bc} = M_b\cdot\ell^{ab}$. Consequently, $(Q^b)' = M_b\cdot(Q^a)'$.

After having extended $Q^a$ to the vicinity of $b$ (through $\overline{\mathbf{H}_+}$), we may calculate $M_b$ from~the property $M_b\cdot(Q^a)' = P\I P\II$, where $(P\I P\II)_{j,k}\coloneqq P_j\I\cdot P_k\II$ for $\ell_1\J,\ell_2\J$ adapted to the interval $bc$. In particular, this shows that, in the calculations done in this section, we do not have~to~solve $\Phi$ at $b$ (which might take time in practice), but rather simply compute $P\I P\II$ at $4$ points near $b$ to calculate $M_b$ and then use the formula $Q^b(z) = M_b\cdot(Q^a(z)-Q^a(b))$. After that, we proceed with our work in the interval $bc$ as usual.
\smallskip

As a technical detail: Computing $M_b$ and checking whether its values are integers was a very useful method of checking for programming errors while computing the results of this paper. In examples of Type I and III, repeating this procedure naively yields only a $3\times 3$ matrix (since then $d_0=3$), from which the full $4\times 4$ matrix can be recovered in a straightforward way. 
\end{rem}

Similarly, we define matrices $M_c$, $M_d$, $M_e$ and $M_a$.

Going further, we will have need for the \textit{monodromy matrix} of $(Q^a)'$ at $a$, and similarly for $b,c,d,e$. Indeed, each function $(Q^a_{j,k})' = P_j\I\cdot P_k\II$ is multivalued near $a$. A counter-clockwise turn around $a$ has the effect of replacing $(Q^a)'$ by some $N_a\cdot (Q^a)'$ (since it preserves $\Phi$). The integer matrix $N_a$ is easily calculated, as we can explicitly compute $P_j\I(z)\cdot P_k\II(z)$ at any~given~$z$.
We put all of this together in a useful form:

\begin{defn}\label{defexmp2}
To work with a single choice of basis $\ell_1\J,\ell_2\J$, we make the arbitrary choice of adapting all results to the interval $ab$. This yields the vectors
$$
r^{ab}\coloneqq q^{ab},\;\;\;\;\;\;\;\;
r^{bc}\coloneqq M_b^{-1}\cdot q^{bc},\;\;\;\;\;\;\;\;
r^{cd}\coloneqq  M_b^{-1}\cdot M_c^{-1}\cdot q^{cd},\;\;\;\;\;\;\;\;
\ldots
$$
and the monodromy matrices
$$
\Theta_a\coloneqq N_a,\;\;\;\;\;\;\;\;
\Theta_b\coloneqq M_b^{-1}\cdot N_b\cdot M_b,\;\;\;\;\;\;\;\;
\Theta_c\coloneqq M_b^{-1}\cdot M_c^{-1}\cdot N_c\cdot M_c\cdot M_b,\;\;\;\;\;\;\;\;
\ldots
$$
which is the data we need in the following sections (see Remark \ref{remsecondmain}).
\end{defn}

\begin{rem}
Note that these results can already give us some information on the total periods: Suppose that a $3$-cycle $\gamma$ in $X|_{\mathbf P^1\setminus\Sigma}$ lies above a loop $\alpha\subseteq\mathbf P^1\setminus\Sigma$ which winds once around $b$, twice around $c$, then five times around $a$. Equivalently, there exists a vector $v\in\mathbf Z^4$ such that $v = (\Theta_b\Theta_c^2\Theta_a^5)^\top\cdot v$. The associated period is seen to be equal to the following integral
$$\int_\gamma\widetilde{\omega} = \big\langle v,\;\; r^{ab}+\Theta_br^{bc}+\Theta_b\Theta_c^2(-r^{ab}-r^{bc})\big\rangle, \;\;\textrm{ for the scalar product }\langle\,,\rangle:\mathbf Z^4\times\mathbf C^4\rightarrow\mathbf C^4$$
by tightening the loops around $a,b,c$. This will be greatly generalized in the following section, where we find a systematic way to express all periods of $X$ and $\widehat{X}$.
%Any product of the following form
%$$\Theta_{prod}\coloneqq \Theta_B^{n_2}\Theta_C^{n_3}\ldots \Theta_A^{n_{m+1}}\Theta_B^{n_{m+2}}\Theta_C^{n_{m+3}}\ldots \Theta_A^{n_{2m+1}}\Theta_B^{n_{2m+2}}\Theta_C^{n_{2m+3}}\ldots \Theta_A^{n_1}$$
%and vector $v\in\mathbf Z^4$ such that $\Theta_{prod}\cdot v = v$ together define a period of $\widetilde{\omega}$ by the inner product:
%$$\left\langle v,\;\;\;\; \Theta_B^{n_2}r^{BC} + \Theta_B^{n_2}\Theta_C^{n_3}r^{CD} + \ldots + \left(\Theta_B^{n_2}\Theta_C^{n_3}\ldots \Theta_A^{n_{m+1}}\Theta_B^{n_{m+2}}\right)r^{BC} + \ldots + r^{AB}\right\rangle$$
%Indeed, this corresponds to integrating over a $3$-chain $\gamma$ described by translating a representative of $\sum_{j,k}v_{j,k}(\ell_j\I\otimes\ell_k\II)\in\mathrm{H}_2(X_z)$ ($z\in\mathbf H_+$) $n_2$ times around $b$, $n_3$ times around $c$, etc. The condition $\Theta_{prod}\cdot v = v$ means that $\gamma$ is a $3$-cycle (up to a $3$-chain in ).

It is reassuring to check that, as one may expect, $M_eM_dM_cM_bM_a = I$ and $\Theta_a\Theta_b\Theta_c\Theta_d\Theta_e = I$. Also, $r^{ab}+r^{bc}+r^{cd}+r^{de}+r^{ea} = 0$ and $r^{ab}+\Theta_br^{bc}+\Theta_b\Theta_cr^{cd}+\Theta_b\Theta_c\Theta_dr^{de}+\Theta_b\Theta_c\Theta_d\Theta_er^{ea} = 0$ (corresponding to trivial loops in the upper and lower half-plane, respectively). In practice,~this check prevents banal calculation errors.
\end{rem}
%\bigskip
%\bigskip

\begin{rem}
On our PC, the procedure outlined in this section took 3-4 minutes per interval~on average, for up to 20 minutes in total. For the interval $ab$, the single most time-consuming step is the initial calculation of the four functions $P_j\I P_k\II$ at the points $a+\varepsilon_i$, $i=1,2,3,4$, often taking more than 1 minute (cf.\ Remark \ref{remuphpdeform}) since it requires the precise computation of elliptic integrals. %(If one is willing to forego the added information obtain by the matrices $M_s$, one may skip this step for the remaining intervals and work with values adapted to the interval $ab$ from the beginning.)
Afterwards, the solutions of the differential equation obtained as expansions~around the starting point $a$ using the (very fast) Frobenius algorithm must be integrated term-by-term, which is again comparatively slow. This expansion-and-integration~step repeats 3 times per interval on average (once for every numerical holomorphic continuation, to obtain $G_1,G_2,G_3,G_4$), which essentially accounts for the remainder of the runtime.
%The total procedure outlined in this section takes less than 15 minutes: The~single most time-consuming step is the initial calculation of the four functions $P_j\I P_k\II$ at the points $a+\varepsilon_i$, $i=1,2,3,4$, usually taking several minutes (cf.\ Remark \ref{remuphpdeform}) since it requires the precise computation of elliptic integrals. After that, the solutions of the differential equation obtained as expansions around the starting point $a$ using the (very fast) Frobenius algorithm must be integrated term-by-term, which is again comparatively slow. This expansion-and-integration~step repeats around 20 times in total (once for every numerical holomorphic continuation, to obtain $G_1,G_2,G_3,G_4$), which essentially accounts for the remainder of the runtime.
\end{rem}

To summarize, we have obtained vectors $q^{ab},\ldots,q^{ea}$ and transformation matrices $M_a,\ldots,M_e$, then produced vectors $r^{ab},\ldots,r^{ea}$ and monodromy matrices $\Theta_a,\ldots,\Theta_e$ (all of them adapted to the interval $ab$). For the explicit example considered, we have provided this data at:\\
\centerline{\url{https://github.com/donlagic-azur/numcalc-schoen-supplement/tree/main/sect4exmp}}
%\smallskip
%\bigskip

%\input{sect4q-ex}
%\input{sect4r-ex}

\section{Determining the $\mathbf Z$-Module of Periods}

Going forward, we will be more interested in the singular fibers of the threefolds $X$ and $\widehat{X}$. Recall the following crucial statement (stated as Proposition C.11 in \cite{PS08} without proof):

\begin{prop}\label{propdefret}
Let $\mathbf D\subseteq\mathbf C$ denote the unit disk and let $Y\rightarrow\mathbf D$ be a proper holomorphic map which is smooth over $\mathbf D^*\coloneqq\mathbf D\setminus\{0\}$. Then there is a strong deformation retract of~$Y$~onto~$Y_0$.
\end{prop}
\begin{prf}
When $Y$ is smooth and $Y_0$ is a normal crossing divisor, the statement is classically shown in \cite[\textsection 3]{Cle77}. Otherwise, we may reduce to this case by a \textit{log resolution} ${\sigma : \widetilde{Y}\rightarrow Y}$~(see~\cite[Theorem 5.4.2]{AHV18}), an isomorphism away from $\widetilde{Y}_0 = \sigma^{-1}(Y_0)$, which is a normal crossing~divisor.
If $F : \widetilde{Y}\times I\rightarrow\widetilde{Y}_0$ is a strong deformation retract, then we\footnote{This proof was written following the post of \textit{AG learner} at: \url{ https://math.stackexchange.com/a/4401330}}
define a map $G : Y\times I\rightarrow Y_0$ by:
$$G(y,t)\coloneqq\left\{\begin{aligned}
    &\sigma(F(\sigma^{-1}(y),t)) & & \textrm{ if }y\notin Y_0\\
    &y & & \textrm{ if }y\in Y_0
\end{aligned}\right.$$
It is easily seen to be continuous, therefore also a strong deformation retract.
\end{prf}

Now, consider our usual setup: A threefold $X = E\I\times_{\mathbf P^1}E\II$, where $E\J\rightarrow\mathbf P^1$ is an elliptic surface with all singular fibers semistable and over $\Sigma\J$. Let $\Sigma\coloneqq \Sigma\I\cup\Sigma\II$. We simultaneously study both $X$ and its small resolution $\widehat{X}$ by considering $Y\in\{X,\widehat{X}\}$. The surfaces $E\I,E\II$~are assumed to have sections, denoted by $\textbf 0\I,\textbf 0\II$ respectively. %If the regular fibers are presented in Weierstrass form, these sections determine exactly the point at infinity of each such fiber. 

\begin{rem}\label{remseclift}
The section $\textbf 0\J$ intersects transversely each fiber $E\J_z$, $z\in\mathbf P^1$, hence it does not pass through any of the nodes in any singular fiber $E\J_s$, $s\in\Sigma$.
This shows that both of the surfaces $\textbf 0\I{\times_{\mathbf P^1}}E\II$ and $E\I{\times_{\mathbf P^1}}\textbf 0\II$ lift (uniquely) from $X$ to $\widehat{X}$, since the small resolution is an isomorphism everywhere away from products of two nodes.
Therefore, there exist sections $\iota\I,\iota\II$ of the projections $\mathrm{pr}\J : Y\rightarrow E\J$ for $Y\in\{X,\widehat{X}\}$. In particular, the composition
$$\mathrm{H}_k(E\I_z)\oplus\mathrm{H}_k(E\II_z)\xrightarrow{\iota\I_*+\iota\II_*}\mathrm{H}_k(Y_z)\xrightarrow{\mathrm{pr}\I_*,\,\mathrm{pr}\II_*}\mathrm{H}_k(E\I_z)\oplus\mathrm{H}_k(E\II_z)$$
is the identity map on the fibers over any point $z\in\mathbf P^1$. Induced by the resulting projection
$$\pi_z \;\; :\;\; \mathrm{H}_k(Y_z)\xrightarrow{\mathrm{pr}\I_*,\,\mathrm{pr}\II_*}\mathrm{H}_k(E\I_z)\oplus\mathrm{H}_k(E\II_z)\xrightarrow{\iota\I_*+\iota\II_*}\mathrm{H}_k(Y_z)$$
is a splitting $\mathrm{H}_k(Y_z) = \mathrm{H}'_k(Y_z)\oplus\pi_z(\mathrm{H}_k(Y_z))$. For example, $\mathrm{H}'_2(X_z)\cong\mathrm{H}_1(E_z\I)\otimes\mathrm{H}_1(E_z\II)$ by the K\"unneth formula (cf.~Example~\ref{exmpdefgm}). The corresponding local system over $\mathbf P^1\setminus\Sigma$~splits~similarly.
\end{rem}
\smallskip

\begin{defn}\label{defnvancyc}
Fix a point $p\in\mathbf P^1\setminus\Sigma$ and, for each $s\in\Sigma$, an open path $ps\subseteq\mathbf P^1\setminus\Sigma$ such that no two paths intersect (drawn below in Figure \ref{diag3}(a)). Applying Proposition \ref{propdefret} to an open neighborhood $\mathcal D$ of the closed path $\overline{ps}$ for $s\in\Sigma$, we get a composition:
$$\mathrm{H}_2(Y_p)\rightarrow\mathrm{H}_2(Y|_{\mathcal D})\cong\mathrm{H}_2(Y_s)$$
We denote the kernel of this map by $VC_s = VC_s(Y)$ and call its elements the \textit{vanishing cycles} at $s$. The map is functorial with respect to the projections $Y\rightarrow E\J$, hence we have a splitting $VC_s = VC'_s\oplus\pi_p(VC_s)$ by Remark \ref{remseclift}.
\end{defn}

Recall our fixed $3$-form $\widetilde{\omega}$ on $Y$ from Section 2. Denote by $\mathcal I : \mathrm H_3(Y)\rightarrow\mathbf C$ the map given by $\mathcal I([\gamma]) = \int_\gamma\widetilde{\omega}$. We are interested in the finite-rank $\mathbf Z$-module image $\mathrm{im}(\mathcal I)\subseteq\mathbf C$. The importance of introducing vanishing cycles comes from the following observation: Pick $\ell_s\in VC_s$ for each $s\in\Sigma$ such that $\sum_s\ell_s=0$ in $\mathrm H_2(Y_p)$. Then the following sum
$$\mathcal I_0\big((\ell_s)_{s\in\Sigma}\big)\coloneqq\sum\nolimits_{s\in\Sigma} q^{ps}_{\ell_s}
\;\;\textrm{ (where }
q^{ps}_{\ell_s}\coloneqq\int_{\ell_s\times ps}\widetilde{\omega}
\;\textrm{ as in Definition \ref{defnpartres}, except }p\notin\Sigma\textrm{)}$$
is an element of $\mathrm{im}(\mathcal I)$. Indeed, we argue as follows:
\bigskip

For each point $\tilde{s}\in\Sigma$, the $3$-chain $\ell_{\tilde{s}}\times p\tilde{s}$ represents some class in the relative homology group $\mathrm H_3(Y,\;Y_p\sqcup\bigsqcup_{s\in\Sigma}Y_s)$. Considering the long exact sequence in relative homology
\begin{center}\begin{tikzcd}
    \mathrm{H}_3(Y_p)\oplus\bigoplus
    \nolimits_s\mathrm{H}_3(Y_s)\arrow[r] & \mathrm H_3(Y)\arrow[r] & \mathrm H_3\left(Y,\;Y_p\sqcup\bigsqcup_sY_s\right)\arrow[r, "\delta"] & \mathrm{H}_2(Y_p)\oplus\bigoplus
    \nolimits_s\mathrm{H}_2(Y_s)
\end{tikzcd}\end{center}
we see that the sum $\sum_s[\ell_s\times ps]$ is in $\ker\delta$ by definition of $\ell_s$ and of vanishing cycles. It thus admits a preimage $[\gamma]\in\mathrm H_3(Y)$ and we conclude $\mathcal I([\gamma]) = \mathcal I_0((\ell_s)_{s\in\Sigma})$ since integration of $\widetilde{\omega}$ over $3$-cycles contained in individual fibers is trivial (Remark \ref{remint3formvanish}).

Note that we only define $\mathcal I_0$ on the kernel of the sum map $\bigoplus\nolimits_sVC_s\rightarrow\mathrm{H}_2(Y_p)$. Moreover, $\mathcal I_0((\pi_p(g_s))_{s\in\Sigma}) = 0$ because, in the notation of Remark \ref{remseclift}, the integral of $\widetilde{\omega}$ is trivial over any $3$-chain contained in $\textbf 0\I{\times_{\mathbf P^1}}E\II$ or $E\I{\times_{\mathbf P^1}}\textbf 0\II$ (again, see Remark \ref{remint3formvanish}). This~shows~that
$$\mathcal I_0\left(\mathrm{ker}\!\left(\bigoplus\nolimits_{s}VC_s\rightarrow\mathrm{H}_2(Y_p)\right)\right) = \mathcal I_0\left(\mathrm{ker}\!\left(\bigoplus\nolimits_{s}VC'_s\rightarrow\mathrm{H}_2(Y_p)\right)\right)$$
%As explained at the end of this section, we are able to effectively calculate the latter set. To show that this is sufficient for our needs, we prove the following theorem.
and we summarize the consequences below.
\medskip

\begin{rem}\label{remsecondmain}
We are able to effectively calculate the latter of these two sets, as this remark will explain. First of all, we have $q^{ps} = q^{pa}+q^{as}$ for a fixed point $a\in\Sigma$, so that
$$\mathcal I_0((\ell_s)_{s\in\Sigma}) \;=\; \sum\nolimits_{s\in\Sigma} q^{ps}_{\ell_s} \;=\;q^{pa}_{\Sigma_s\ell_s}+\sum\nolimits_{s\in\Sigma} q^{as}_{\ell_s}$$
but $\Sigma_s\ell_s = 0$ and $q^{pa}_{\Sigma_s\ell_s} = 0$. Therefore the value $\mathcal I_0((\ell_s)_{s\in\Sigma})$ does not depend on $p$, and in our explicit calculations we may in fact replace $q^{ps}_{\ell_s}$ by $q^{as}_{\ell_s}$. (The assumption $p\notin\Sigma$ is only~convenient for theoretical purposes in this section.)

In other words, the only integrals explicitly needed in this section are exactly the partial periods $q^{as},s\in\Sigma$ which were calculated in Section 3. For convenience of calculation, we will fix some basis of $\mathrm{H}'_2(Y_p)\cong\mathrm{H}_1(E_p\I)\otimes\mathrm{H}_1(E_p\II)$. (In the notation of the example covered in Section 4, we define vectors $q^{ad}\coloneqq r^{ab}+r^{bc}+r^{cd}$, etc.) Apart from this, we will consider the monodromy matrices $\Theta_s$ of $\mathcal H_1(E\I|_{\mathbf P^1\setminus\Sigma},\mathbf Z)\otimes\mathcal H_1(E\II|_{\mathbf P^1\setminus\Sigma},\mathbf Z)$ at $s$ (cf. beginning of Section 3) in this same basis. These are easily calculated; for example, see Definition \ref{defexmp2}.

With the statement of Theorem \ref{thmperMV} as written below, we have now completely described the group $\mathcal I(\mathrm{H}_3(Y))$ of periods of $\widetilde{\omega}$ on $Y$ in terms of vanishing cycles. The algorithm to determine the $\mathbf Z$-module $\mathcal I_0\big(\mathrm{ker}(\bigoplus\nolimits_{s}VC'_s\rightarrow\mathrm{H}_2(Y_p))\big)$ is straightforward:
\begin{enumerate}
    \item Determine the groups $VC'_s\subseteq\mathrm{H}'_2(Y_p)$ (this is done in the next section; we explicitly give $VC'_s$ using the matrix $\Theta_s$ via Proposition \ref{propsingvanish} and Corollary \ref{cor2}). The collection of these vanishing cycles is how our program encodes the difference between the threefolds $X$ and $\widehat{X}$ (so far, our partial periods are the same for both). After taking small resolutions, it should have fewer vanishing cycles, hence fewer periods.
    \item Consider the $\mathbf Z$-module $\bigoplus_{S}VC'_S$. Take the kernel of the map $(\ell_s)\mapsto\sum_s \ell_s$ (its basis~can be found using the Lenstra–Lenstra–Lov\'asz algorithm, \cite{LLL82}), then get a $\mathbf Z$-basis for its image in $\mathbf C$ under the linear map $\mathcal I_0 : (\ell_s)\mapsto\sum_s q^{as}_{\ell_s}$ (this is another application of the LLL algorithm). By Theorem \ref{thmperMV} below, this image is exactly $\mathcal I(\mathrm{H}_3(Y))$.
\end{enumerate}
Given a finite set $F\subseteq\mathbf C$, the application of the LLL algorithm to find a basis for the~group~generated by $F$ is standard, however it is worth pointing out that its success is tied to the precision (number of digits) with which we express the complex numbers in $F$; it might fail to determine relations between numbers expressed to a high precision. In practice, we first set the elements of $F$ to a much smaller precision and then use the LLL algorithm (in our case available~as~part of Maple's \textit{Integer Relations} package) which produces a set $B$ of integer combinations of elements of $F$. The set $B$ is a candidate for the desired basis and, because of the reduced precision, it is certainly a $\mathbf Z$-independent set inside of $\mathbf C$. We now compute the obtained integer combinations at a higher precision, then confirm that $F$ is indeed contained in the $\mathbf Z$-span of $B$. Our precision of 30 decimal places is sufficient to identify a correct basis, as in all considered cases the integer combinations in the construction of $B$ have coefficients with no more than 2 digits.
\end{rem}

\begin{thm}\label{thmperMV}
All periods of $Y$ are of the form described above. That is:
$$\mathcal I\,\big(\mathrm{H}_3(Y)\big) = \mathcal I_0\left(\mathrm{ker}\!\left(\bigoplus\nolimits_{s}VC_s\rightarrow\mathrm{H}_2(Y_p)\right)\right)$$
\end{thm}

\begin{proof}
We start by first studying $\mathrm{H}_3(Y,Y_p)$. Since $Y_p$ is a deformation retract of its neighborhood (by Proposition \ref{propdefret}), this group is canonically isomorphic to $\mathrm{H}_3(Y/Y_p)$, where we write $Y/Y_p$ for the topological quotient space in which all points of $Y_p$ are mutually identified.

Pick a point $q\in\mathbf P^1$ such that there exists a simply connected open neighborhood $\mathcal V$~of~$q$ disjoint from all $ps$. For each $s\in\Sigma$ we also fix a simply connected open neighborhood $\mathcal U_s$ of~the closed path $\overline{ps}$ such that the following holds (see illustration below, Figure \ref{diag3}(b)):
\begin{itemize}
    \item the sets $\mathcal V$ and (all) $\mathcal U_s$ cover $\mathbf P^1$, i.e. $\mathbf P^1 = \mathcal V\cup\bigcup_{s\in\Sigma}\mathcal U_s$
    \item for any $\Sigma'\subseteq\Sigma$ and $s\in\Sigma\setminus\Sigma'$, the open neighborhood $\mathcal U_s\cap\bigcup_{t\in\Sigma'}\mathcal U_t$ is contractible to $p$
\end{itemize}

\begin{figure}[H]
  \centering
  \captionsetup{width=\linewidth}
  \includegraphics[width=\linewidth]{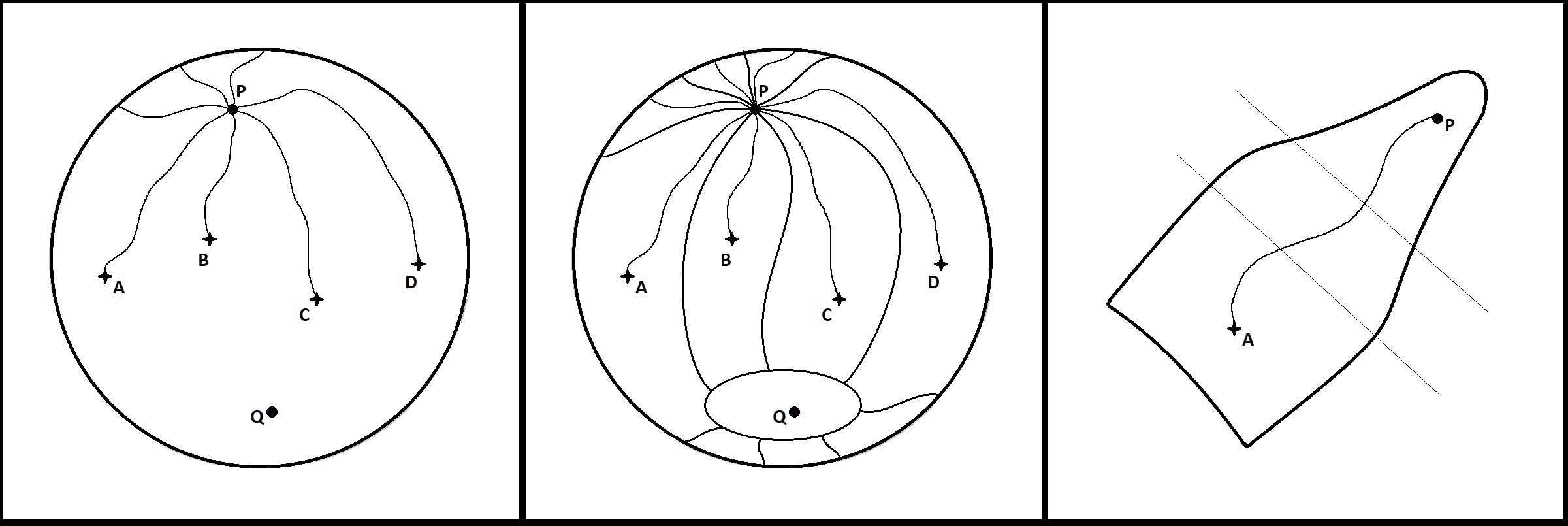}
  \caption{(images a,b,c; from left to right)
  \textit{
  Up to homeomorphism, the points in $\Sigma\sqcup\{p,q\}$ and paths $ps$ look as in image (a). The second, image (b), shows a possible cover of $\mathbf P^1$ by $\mathcal V$ and $\mathcal U_s$ as described above. In particular, it is clear that such a cover exists (homeomorphically to this image). Finally, image (c) shows a cover of the set $\mathcal U_a$, as below.
  }}\label{diag3}
\end{figure}

Cover $\mathcal U_s$ by a simply connected neighborhood of $s$ and a simply connected neighborhood of $p$, with simply connected intersection (over which every fiber is generic and identified up to homotopy with $Y_p$), as in Figure \ref{diag3}(c). Proposition \ref{propdefret} and the Mayer-Vietoris theorem, applied to the respective cover of the space $(Y|_{\mathcal U_s})/Y_p$, then imply that the following sequence is exact:
$$\mathrm{H}_3(Y_p)\rightarrow\mathrm{H}_3(Y_s)\oplus\mathrm{H}_3(Y_p/Y_p)\rightarrow\mathrm{H}_3(Y|_{\mathcal U_s}/Y_p)\rightarrow\mathrm{H}_2(Y_p)\rightarrow\mathrm{H}_2(Y_s)\oplus\mathrm{H}_2(Y_p/Y_p)$$
Observe that, if $(\ell_s\times ps)\in\mathrm{Z}_3(Y|_{\mathcal U_s}/Y_p)$ is any $3$-cycle constructed above $ps$ by extending some $2$-cycle representing $\ell_s\in VC_s$, then the above map 
$$\mathrm{H}_3(Y|_{\mathcal U_s}/Y_p)\longrightarrow\ker\!\big(\mathrm{H}_2(Y_p)\rightarrow\mathrm{H}_2(Y_s)\big) = VC_s$$
sends $[\ell_s\times ps]$ back to $\ell_s$. This is the key element of the compatibility of $\mathcal I$ and $\mathcal I_0$.
\bigskip

%By the definition of $\mathcal U_s$, the space $(Y|U')/Y_p$.
If $\Sigma = \{s_1,\ldots,s_n\}$ and $\mathcal W_m\coloneqq\mathcal U_{s_m}\cap\bigcup_{i<m}\mathcal U_{s_i}$, then $(Y|_{\mathcal W_m})/Y_p$ is contractible by definition~of the sets $\mathcal U_s$. Iterating the Mayer-Vietoris theorem, we get that $\mathrm{H}_3(Y|_{\mathcal U}/Y_p)\cong\bigoplus\nolimits_s\mathrm{H}_3(Y|_{\mathcal U_s}/Y_p)$ for $\mathcal U = \bigcup_{s\in\Sigma}\mathcal U_s$. Now the top row of the following diagram is exact by the previous paragraph:
\vspace{-15pt}

\begin{center}\begin{tikzcd}
    & \bigoplus\nolimits_s\mathrm{H}_3(Y_s)\arrow[r]\arrow[d] & \bigoplus\nolimits_s\mathrm{H}_3(Y|_{\mathcal U_s}/Y_p)\arrow[r]\arrow[d, "\rotatebox{90}{$\sim$}", pos=0.4] & \bigoplus\nolimits_s VC_s\arrow[r]\arrow[d] & 0\\
    \mathrm{H}_3(Y_p)\arrow[r] & \mathrm{H}_3(Y|_{\mathcal U})\arrow[r] & \mathrm{H}_3(Y|_{\mathcal U}/Y_p)\arrow[r] & \mathrm{H}_2(Y_p)
\end{tikzcd}\end{center}
\vspace{-3pt}

\noindent
The bottom row comes from the long exact singular homology sequence of the pair $(Y|_{\mathcal U}, Y_p)$~and the vertical maps are given by summation. We deduce that the following sequence~is~exact
$$\mathrm{H}_3(Y_p)\oplus\bigoplus\nolimits_s\mathrm{H}_3(Y_s)\longrightarrow\mathrm{H}_3(Y|_{\mathcal U})\longrightarrow\mathrm{ker}\!\left(\bigoplus\nolimits_{s}VC_s\rightarrow\mathrm{H}_2(Y_p)\right)\longrightarrow 0$$
and that, by construction, the second map commutes with the definitions of $\mathcal I$ and $\mathcal I_0$. Indeed, its definition agrees with our comparison of $\mathcal I$ and $\mathcal I_0$ at the beginning of this section.
\bigskip

It remains only to show that $\mathcal I(\mathrm{H}_3(Y)) = \mathcal I(\mathrm{H}_3(Y|_{\mathcal U}))$. For this, note that $\mathcal U$ and $\mathcal V$ cover $\mathbf P^1$. The intersection is an annulus $\mathcal U\cap\mathcal V$ with $Y|_{\mathcal U\,\cap\,\mathcal V}$ homotopy equivalent to $Y_q\times\mathbf S^1$.~Thus
$$\mathrm{H}_2(Y_q\times\mathbf S^1)\cong(\mathrm{H}_2(Y_q)\otimes\mathrm{H}_0(\mathbf S^1))\oplus(\mathrm{H}_1(Y_q)\otimes\mathrm{H}_1(\mathbf S^1))\cong\mathrm{H}_2(Y_q)\oplus\mathrm{H}_1(Y_q)$$
by the K\"unneth formula (these homology groups are torsion-free since $Y_q$ is a product of tori). Applying the Mayer-Vietoris theorem one more time, to $Y|_{\mathcal U}\cup Y|_{\mathcal V}$, we get the exact sequence:
$$\mathrm{H}_3(Y_q)\oplus\mathrm{H}_3(Y|_{\mathcal U})\rightarrow\mathrm{H}_3(Y)\rightarrow
\mathrm{H}_2(Y_q\times\mathbf S^1)\rightarrow
\mathrm{H}_2(Y_q)\oplus\mathrm{H}_2(Y|_{\mathcal U})$$
The component $\mathrm{H}_2(Y_q\times\mathbf S^1)\rightarrow\mathrm{H}_2(Y|_{\mathcal V})\cong
\mathrm{H}_2(Y_q)$ of the rightmost map is given by inclusions, hence its kernel is canonically identified with $\mathrm{H}_1(Y_q)$. We get a map $\varphi$ and an exact sequence:
$$\mathrm{H}_3(Y_q)\oplus\mathrm{H}_3(Y|_{\mathcal U})\longrightarrow\mathrm{H}_3(Y)\xrightarrow{\;\;\varphi\;\;}
\mathrm{H}_1(Y_q)$$
Now, the K\"unneth formula tells us that $\mathrm{H}_1(Y_q)\xrightarrow{\mathrm{pr}\I_*,\,\mathrm{pr}\II_*}\mathrm{H}_1(E\I_q)\oplus\mathrm{H}_1(E\II_q)$ is an isomorphism. This implies that $\pi_q : \mathrm{H}_1(Y_q)\rightarrow\mathrm{H}_1(Y_q)$ from Remark \ref{remseclift} is the identity map. Let $g\in\mathrm{H}_3(Y)$ be arbitrary. As the construction of $\varphi$ commutes with the projections $\mathrm{pr}\J : Y\rightarrow E\J$, we have
$$\varphi(g) = \pi_q(\varphi(g)) = \varphi(\pi(g))$$
where $\pi : \mathrm{H}_3(Y)\rightarrow\mathrm{H}_3(Y)$ is the composition $(\iota\I_*+\iota\II_*)\circ(\mathrm{pr}\I_*,\,\mathrm{pr}\II_*)$.
In particular, this implies that $g-\pi(g)\in\ker(\varphi)=\mathrm{im}(\mathrm{H}_3(Y_q)\oplus\mathrm{H}_3(Y|_{\mathcal U}))$ and thus $\mathcal I(g-\pi(g))\in\mathcal I(\mathrm{H}_3(Y|_{\mathcal U}))$.
However, $\mathcal I(\pi(g)) = 0$ because, again by Remark \ref{remint3formvanish}, the integral of $\widetilde{\omega}$ is $0$ over all $3$-chains contained in $\textbf 0\I{\times_{\mathbf P^1}}E\II$ or $E\I{\times_{\mathbf P^1}}\textbf 0\II$. Therefore $\mathcal I(g)\in\mathcal I(\mathrm{H}_3(Y|_{\mathcal U}))$.
\end{proof}
\section{Computing the Vanishing Cycles}

We keep the notation of the previous section, notably $VC'_s(Y)$ from Definition \ref{defnvancyc}. Let $s\in\Sigma$. By translation, we may assume that $s = 0$. Consider the following diagram with exact rows
\begin{center}\begin{tikzcd}
    0\arrow[r] & VC'_0(\widehat{X})\arrow[r]\arrow[d, hook] & \mathrm{H}'_2(\widehat{X}_p)\arrow[r]\arrow[d, equal] & \mathrm{H}_2(\widehat{X}_0)\arrow[d]\\
    0\arrow[r] & VC'_0(X)\arrow[r] & \mathrm{H}'_2(X_p)\arrow[r] & \mathrm{H}_2(X_0)
\end{tikzcd}\end{center}
where we have used on the right that $\mathrm{H}'_2(Y_0)\hspace{-1pt}\hookrightarrow\hspace{-1pt}\mathrm{H}_2(Y_0)$. We want to compute the~groups~$VC'_0(\widehat{X})$ and $VC'_0(X)$. This calculation can be done in a disk $\mathcal D$ around $0$ in $\mathbf P^1$ in which we fix $p\neq 0$.
%It's easier in the case of $VC'_0(X)$, since
%$$VC'_0(X) = (VC_0(E\I)\otimes\mathrm{H}_1(E\II_p))+(\mathrm{H}_1(E\II_p)\otimes VC_0(E\II))\subseteq \mathrm{H}'_2(X_p)$$
%where %$VC_0(E\J)\coloneqq\ker\big(\mathrm{H}_1(E\J_p)\rightarrow\mathrm{H}_1(E\J_0)\big)$.
Moreover, we will want to express these results in terms intrinsic to the monodromy operator~$\Theta_0$ on $\mathcal H_1(E\I|_{\mathbf P^1\setminus\Sigma},\mathbf Z)\otimes\mathcal H_1(E\II|_{\mathbf P^1\setminus\Sigma},\mathbf Z)$ (see Remark \ref{remsecondmain}), to be used in any chosen basis.

\begin{defn}
Let $G$ be an Abelian group and $H\subseteq G$ its subgroup. Then we define
$$\,\overline{\! H}\coloneqq\{g\in G\mid (\exists k\neq 0)\,kg\in H\}$$
and call $\,\overline{\! H}$ the \textit{saturation} of $H$, where the ambient group $G$ is always clear in context.
\end{defn}

\begin{exmp}\label{exmptormonod}
Consider an elliptic surface $E\rightarrow\mathbf P^1$ with a type-$I_n$ semistable singular fiber above $0$. The fiber is topologically a torus pinched at $n$ places (as in Figure \ref{figptor} below) and the monodromy matrix acting on $\mathrm{H}_1(E_p)$ is of the following form in some $\mathbf Z$-basis $(u,v)$ of $\mathrm{H}_1(E_p)$:
\begin{center}
$\Theta = \Theta_{E,0}\coloneqq\begin{bmatrix}
    1 & n\\
    0 & 1
\end{bmatrix}$\;\;\;\;\;\;\;\;\;\;
\end{center}
Let $VC_0(E)\coloneqq\ker\big(\mathrm{H}_1(E_p)\rightarrow\mathrm{H}_1(E_0)\big)$.
For any class $w\in\mathrm{H}_1(E_p)$, the image of $\Theta w-w$ in  the set $\mathrm{H}_1(E|_\mathcal D)$ is $0$ since it's represented by the boundary of the $2$-chain that some representative $1$-cycle of $w$ describes as it encircles $0$ once. By Proposition $\ref{propdefret}$, this gives $\Theta w-w\in VC_0(E)$. Thus $n\mathbf Z\cdot u = \mathrm{im}(\Theta-I)\subseteq VC_0(E)$. Moreover, because the group $\mathrm{H}_1(E_0)\simeq\mathbf Z$ is torsion-free, $\mathbf Z\cdot u = \overline{\mathrm{im}(\Theta-I)}\subseteq VC_0(E)$. On the other hand, the subgroup $\mathbf Z\cdot v$ gets mapped onto $\mathrm{H}_1(E_0)$, hence $VC_0(E) = \mathbf Z\cdot u$. A cycle in $\mathrm{H}_1(E_p)$ is vanishing at $0$ if and only if it's a multiple of $u$.

\begin{figure}[H]
  \centering
  \captionsetup{width=\linewidth}
  \includegraphics[height=60 pt]{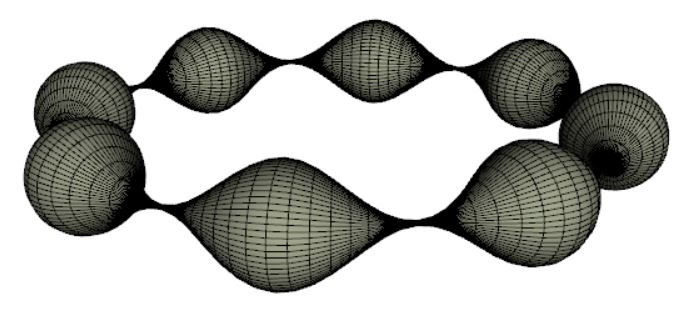}
  \caption{\textit{An $8$-pinched torus.}}\label{figptor}
\end{figure}
\end{exmp}

Now, consider $X = E\I \!\times_{\mathbf P^1}\! E\II$ and suppose that the elliptic surfaces $E\I,E\II$ have singular fibers of type $I_n,I_m$ at $0$, respectively. Let $(u_1,v_1)$ and $(u_2,v_2)$ be as above in Example \ref{exmptormonod}. Then $$(uu,uv,vu,vv)\coloneqq (u_1\otimes u_2,u_1\otimes v_2,v_1\otimes u_2,v_1\otimes v_2)$$ is a basis of $\mathrm{H}'_2(X_p)$. As a consequence of the previous example, the group $\mathbf Z\cdot vv$ gets mapped onto $\mathrm{H}'_2(X_0)\simeq\mathbf Z$, hence $uu,uv,vu\in VC'_0(X)$.

The case when $nm = 0$ also appears in our situation, where type $I_0$ denotes a regular fiber. If $n = 0$, then $E\I$ has no singularity at $0$ and no vanishing cycles. It follows that $uu, vu$ are vanishing, but $uv, vv$ map onto a basis of $\mathrm{H}'_2(X_0)\simeq\mathbf Z^2$. Similarly when $m = 0$. This discussion can be summed up in the following result:\newpage

%It is useful, for our considerations, to have a description of $VC'_0(X)$ which is not dependent on a basis of $\mathrm{H}'_2(X_p)$. The above discussion can be summed up in the following result:

\begin{prop}\label{propsingvanish}
Let the elliptic surfaces $E\I,E\II$ have singular fibers of type $I_n,I_m$ at $0$, where $n,m\geq 0$. Let $\Theta = \Theta_{X,0} = \Theta_{E\I,0}\otimes \Theta_{E\II,0}$ be the monodromy operator on $\mathrm{H}'_2(X_p)$. Then:
$$VC'_0(X) \;=\; \mathrm{ker}(\Theta-I)^{1+\mathrm{rank}((\Theta-I)^2)} \;=\; \left\{
\begin{aligned}
    &\mathrm{ker}(\Theta-I)^2 & &\textrm{if }nm\neq 0\\
    &\mathrm{ker}(\Theta-I)& &\textrm{if }nm = 0
\end{aligned}\right.$$
\end{prop}
\begin{proof}
The equalities are easily seen in the basis $(uu,uv,vu,vv)$, in which we have:
$$\Theta = \begin{bmatrix}
    1 & m & n & nm\\
    0 & 1 & 0 & n\\
    0 & 0 & 1 & m\\
    0 & 0 & 0 & 1
\end{bmatrix},\;\;\;\;
\Theta-I = \begin{bmatrix}
    0 & m & n & nm\\
    0 & 0 & 0 & n\\
    0 & 0 & 0 & m\\
    0 & 0 & 0 & 0
\end{bmatrix},\;\;\;\;
(\Theta-I)^2 = \begin{bmatrix}
    0 & 0 & 0 & 2nm\\
    0 & 0 & 0 & 0\\
    0 & 0 & 0 & 0\\
    0 & 0 & 0 & 0
\end{bmatrix}
\;\;\;\;\;\;\;\;
$$
\vspace{-25pt}

\end{proof}
\medskip

This settles the case of $VC'_0(X)$. It remains to compute $VC'_0(\widehat{X})\subseteq VC'_0(X)$. If $nm = 0$,~then these two sets agree and the result has already been computed above. Thus, assume $nm\neq 0$ and write $d\coloneqq\gcd(m,n) = xm+yn$ for some choice of $x,y\in\mathbf Z$. The monodromy matrix~$\Theta$~is
$$
\begin{bmatrix}
    1 & m & n & nm\\
    0 & 1 & 0 & n\\
    0 & 0 & 1 & m\\
    0 & 0 & 0 & 1
\end{bmatrix}\!\!\textrm{, which becomes }\!
\begin{bmatrix}
    1 & 2nm/d & yn-xm & nm\\
    0 & 1 & 0 & d\\
    0 & 0 & 1 & 0\\
    0 & 0 & 0 & 1
\end{bmatrix}\!\textrm{ in basis }\!
\begin{bmatrix}
    uu\\
    n/d\cdot uv + m/d\cdot vu\\
    y\cdot vu - x\cdot uv\\
    vv
\end{bmatrix}
$$
of the group $\mathrm{H}'_2(X_p)$. The three numbers $2nm/d, yn-xm, nm$ are all divisible by $d$, therefore:
$$\mathrm{im}(\Theta-I)\subseteq \big(d\mathbf Z\cdot uu\big) \oplus \big(d\mathbf Z\cdot (n/d\cdot uv + m/d\cdot vu)\big)$$
In fact, this inclusion is an equality. It suffices to observe $d\cdot uu\in\mathrm{im}(\Theta-I)$, which~holds~by:
$$y\frac{n}{d}-x\frac{m}{d}\equiv y\frac{n}{d}+x\frac{m}{d}\equiv 1\!\!\!\!\mod\!2\textrm{,\; so }\gcd\!\left(\frac{2nm}{d},yn-xm\right) = d\cdot\gcd\!\left(2\frac{m}{d}\frac{n}{d},y\frac{n}{d}-x\frac{m}{d}\right) = d$$
We thus deduce the following equality (and inclusion, via the same argument as in Example~\ref{exmptormonod}), 
$$\big(d\mathbf Z\cdot uu\big) \oplus \big(d\mathbf Z\cdot (n/d\cdot uv + m/d\cdot vu)\big) = \mathrm{im}(\Theta-I) \subseteq VC'_0(\widehat{X})$$
however we cannot conclude that $\overline{\mathrm{im}(\Theta-I)} \subseteq VC'_0(\widehat{X})$ because the group $\mathrm{H}_2(\widehat{X}_0)$ is in general not torsion-free (contrary to the previously considered cases). Indeed, by the following theorem,
$$\mathrm{im}(\Theta-I)\subseteq VC'_0(\widehat{X})\subseteq \overline{\mathrm{im}(\Theta-I)}$$
and neither inclusion is an equality if $d\neq 1$.
\medskip

\begin{thm}\label{thmvanish}
The subgroup $VC'_0(\widehat{X})\subseteq\mathrm{H}'_2(\widehat{X}_0)$ of cycles vanishing at $0$ is exactly:
$$VC'_0(\widehat{X}) = \big(\mathbf Z\cdot uu\big)+\mathrm{im}(\Theta-I) = \big(\mathbf Z\cdot uu\big) \oplus \big(d\mathbf Z\cdot (n/d\cdot uv + m/d\cdot vu)\big)$$
\end{thm}
\begin{rem}\label{remcuriosity}
Recall that Schoen's construction of $\widehat{X}$ (described in Section 2) is dependent on a choice of small resolutions of $X$. The statement of Theorem \ref{thmvanish} is, however, independent of this choice. The essential reason for this is highlighted in the proof of Lemma \ref{lemcalc} below.
\end{rem}
\begin{prf}
The cycle $vv\in\mathrm{H}'_2(\widehat{X}_0)$ is not in $VC'_0(X)$, hence also not in $VC'_0(\widehat{X})$. We now compute the homology group $\mathrm{H}_2(\widehat{X}_0)$ and determine in it the images of $uu,uv,vu$.

First, note that the Betti numbers of an $n$-pinched torus, resp.\ $m$-pinched torus, are $(1,1,n)$, resp.\ $(1,1,m)$, and their homology is torsion-free. Therefore, the fiber $X_0$ has Betti numbers $(1,2,n+m+1,n+m,nm)$ and torsion-free homology by the K\"unneth formula.
\pagebreak

Second, $\widehat{X}$ is (locally above $0$) constructed by taking small resolutions at $nm$ nodes in $X_0$, which replace the $nm$ points by $2$-spheres $(\mathbf S^2_{i,j})_{n,m}$. Hence, gluing the resulting holes in $\widehat{X}_0$ with the $nm$ balls $(\mathbf B^3_{i,j})_{n,m}$ (such that the intersection of each $\mathbf B^3_{i,j}$ with $\widehat{X}_0$ is the sphere $\mathbf S^2_{i,j} = \partial \mathbf B^3_{i,j}$), we get a topological space homotopy equivalent to $X_0$ (by contracting each ball to a point). Consider the following part of the Mayer-Vietoris exact sequence:
$$\mathrm{H}_3\big(X_0\big)\xrightarrow{\;\;\delta\;\;} \mathrm{H}_2\left(\bigsqcup\nolimits_{i,j}\mathbf S^2_{i,j}\right)\xrightarrow{\;\;\alpha\;\;} \mathrm{H}_2\big(\widehat{X}_0\big)\longrightarrow \mathrm{H}_2\big(X_0\big)\longrightarrow 0$$
We know that $uu,uv,vu$ vanish in $\mathrm{H}_2(X_0)$. Thus their images in $\mathrm{H}_2(\widehat{X}_0)$ are in $\mathrm{im}(\alpha) = \mathrm{coker}(\delta)$. Let $S_{i,j}$ be a generator of $\mathrm{H}_2(\mathbf S^2_{i,j})\simeq\mathbf Z$ and $(A_i)^n_{i=1},(B_j)^m_{j=1}$ be the $n+m$ obvious generators of: 
$$\mathrm{H}_3(X_0)\cong(\mathrm{H}_2(E\I_0)\otimes\mathrm{H}_1(E\II_0))\oplus(\mathrm{H}_1(E\I_0)\otimes\mathrm{H}_2(E\II_0))$$
We claim the following (to be proven by explicit calculations in Lemma \ref{lemcalc} below):
\begin{itemize}
    \item The image of $uu$ in $\mathrm{H}_2(\widehat{X}_0)$ is $0$.
    \item $\delta(A_i) = \sum_{j}(S_{i,j}-S_{i+1,j})$ for $i\in\mathbf Z/n$ and $\delta(B_j) = \sum_{i}(S_{i,j}-S_{i,j+1})$ for $j\in\mathbf Z/m$.
    \item The image of $uv$ in $\mathrm{H}_2(\widehat{X}_0)$ is $-\sum_{j}\alpha(S_{i,j})$ (for any fixed $i\in\{1,2,\ldots,n\}$, the~choice~of which doesn't matter by the preceding item) and the image of $vu$ in $\mathrm{H}_2(\widehat{X}_0)$~is~$\sum_{i}\alpha(S_{i,j})$ (for any fixed $j\in\{1,2,\ldots,m\}$).
\end{itemize}
From this, it follows that $uv$ and $vu$ have images $\xi$ and $\eta$ in $\mathrm{im}(\alpha)\subseteq\mathrm{H}_2(\widehat{X}_0)$, which is of~the~form:
$$\mathrm{im}(\alpha)\simeq \frac{\mathbf Z^{nm}}{\mathrm{im}(\delta)} = \frac{\mathbf Z[\xi,\eta,S_{1,1},S_{1,2},\ldots,S_{n,m}]}{\sum_{i}\big\langle\xi+\sum_{j}S_{i,j}\big\rangle+\sum_{j}\big\langle\eta-\sum_{i}S_{i,j}\big\rangle}$$
Finally, it is now completely elementary (and only slightly tedious) to see that $n\xi + m\eta = 0$ is the only relation between these two elements in $\mathrm{im}(\alpha)$. Indeed,
$$n\xi + m\eta = m\sum\nolimits_{i}S_{i,j_0}-n\sum\nolimits_{j}S_{i_0,j} = \sum\nolimits_{j}\sum\nolimits_{i}S_{i,j} - \sum\nolimits_{i}\sum\nolimits_{j}S_{i,j} = 0$$
and any other relation can be reduced to this one by a recursive argument. This ends the proof.\,

To be precise, we prove the above formulas only up to sign in Lemma \ref{lemcalc}. But, then the sign of the generators $A_i,B_j,S_{i,j}$ can just be chosen to make all the formulas true, the~only~exception a priori being the difference in sign between $-\sum_{j}\alpha(S_{i,j})$ and $\sum_{i}\alpha(S_{i,j})$. This difference of signs is somehow encoded in the K\"unneth formula for $\mathrm{H}_2(\widehat{X}_p)$ and $\mathrm{H}_2(X_0)$. However, we see that the two signs must be different, since we already know that $n\cdot uv + m\cdot vu\in\mathrm{im}(\Theta-I)$ is a vanishing cycle and this also follows when the two signs differ. On the other hand,~we~can~conclude~that $n\cdot uv - m\cdot vu$ can't also be vanishing at the same time, because then $2n\sum_{j}\alpha(S_{i,j}) = \pm 2n\cdot uv$ would vanish, but this element is not in the image of $\delta$, which we've computed.
\end{prf}
%\medskip

The following corollary applies to our case, in which $nm = 0$ only if $\{n,m\} = \{0,1\}$. In particular, note that we do not need to know the value of $d$ (although it can always be recovered by additional effort) to compute $VC'_0(\widehat{X})$.

\begin{cor}\label{cor2}
Let the elliptic surfaces $E\I,E\II$ have singular fibers of type $I_n,I_m$ at $0$, where $n,m\geq 0$. Let $\Theta = \Theta_{X,0} = \Theta_{E\I,0}\otimes \Theta_{E\II,0}$ be the monodromy operator on $\mathrm{H}'_2(X_p)$.

If $nm\neq 0$ or $\{n,m\} = \{0,1\}$ (as in Schoen's construction), we get the following equality:
$$VC'_0(\widehat{X}) = \mathrm{im}(\Theta-I)+\overline{\mathrm{im}(\Theta-I)^2}$$
\end{cor}
\begin{proof}
If $nm\neq 0$, then this is an immediate consequence of $\mathrm{im}(\Theta-I)^2 = 2nm\mathbf Z\cdot uu$ and the above theorem. Otherwise $VC'_0(\widehat{X}) = VC'_0(X)$, also with $\mathrm{im}(\Theta-I) = (m+n)\mathrm{ker}(\Theta-I)$ and $(\Theta-I)^2 = 0$. The proof now follows by $m+n = 1$ and Proposition \ref{propsingvanish}.
\end{proof}

\begin{figure}[H]
  \centering
  \captionsetup{width=\linewidth}
  \includegraphics[width=\linewidth]{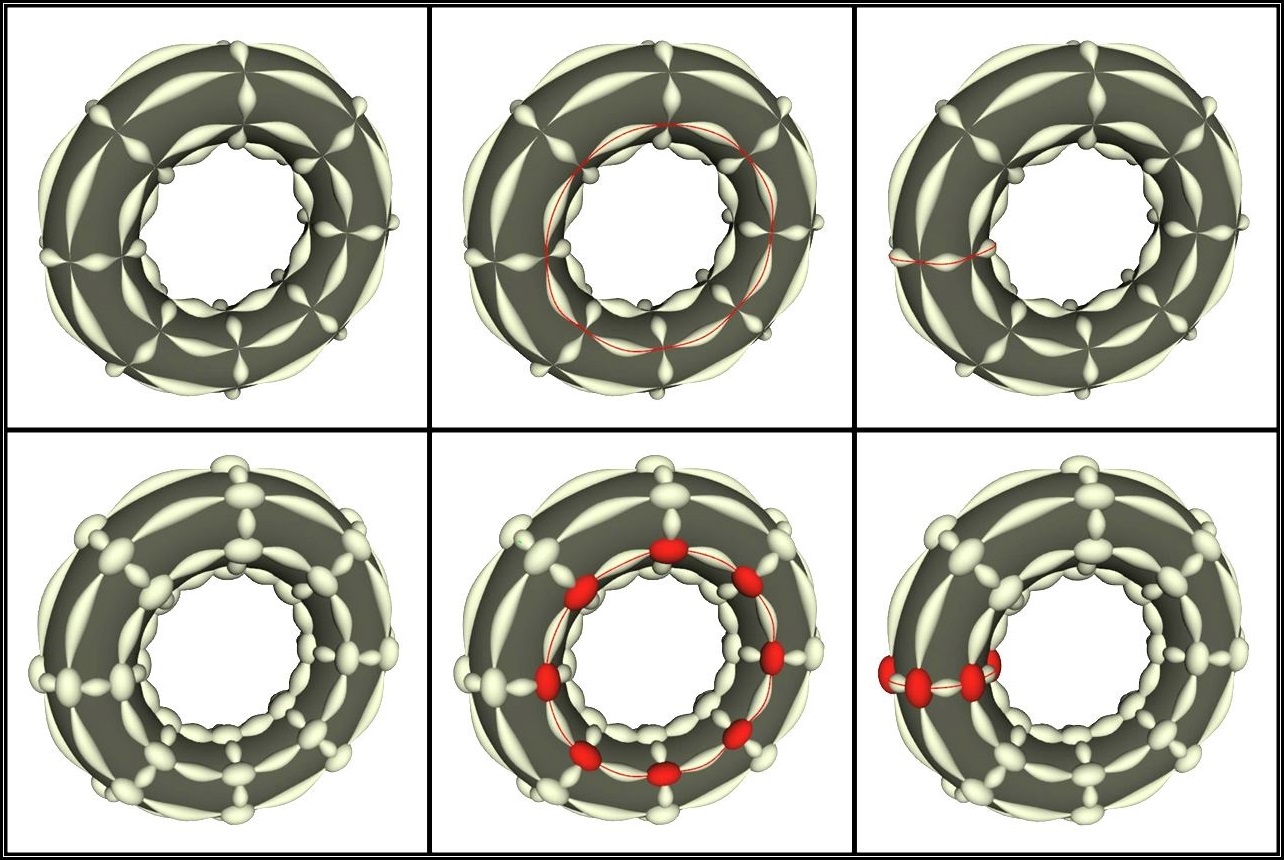}
  \caption{\textit{First row: A schematic representation of $X_0$ with $(n,m) = (6,8)$. Pinched tori have one (nondegenerate) cycle as a dominant feature, so their product naturally resembles a torus. The light areas are those at which one of the two components degenerates to a point, so they are 2-dimensional varieties. The dark areas are locally 4-dimensional and we make no attempt to represent them. The second and third picture highlight natural representatives of $uv$ and $vu$, which are locally $1$-dimensional and whose classes thus clearly vanish.\\
  Second row: A schematic representation of $\widehat{X}_0$, which differs by the replacement of each node by a sphere. The second and third picture again highlight representatives of $uv$ and $vu$, but this time the classes do not vanish since they respectively contain $m$ and $n$ different spheres.}}\label{sing4}
\end{figure}

Finally, we rather explicitly prove the claim from the previous proof:

\begin{lem}\label{lemcalc}
The claim from the proof of Theorem \ref{thmvanish} holds up to signs of $A_i,B_j,uu,uv,vu$:
\begin{enumerate}[\hspace{0.9 cm} a)]
    \item The image of $uu$ in $\mathrm{H}_2(\widehat{X}_0)$ is $0$.
    \item $\delta(A_i) = \sum_{j}(S_{i,j}-S_{i+1,j})$ for $i\in\mathbf Z/n$ and $\delta(B_j) = \sum_{i}(S_{i,j}-S_{i,j+1})$ for $j\in\mathbf Z/m$.
    \item The image of $uv$ in $\mathrm{H}_2(\widehat{X}_0)$ is $-\sum_{j}\alpha(S_{i,j})$ (for any fixed $i\in\{1,2,\ldots,n\}$, the~choice~of which doesn't matter by the preceding item) and the image of $vu$ in $\mathrm{H}_2(\widehat{X}_0)$~is~$\sum_{i}\alpha(S_{i,j})$ (for any fixed $j\in\{1,2,\ldots,m\}$).
\end{enumerate}
Here, the map $\alpha : \mathrm{H}_2(\bigsqcup\mathbf S^2_{i,j})\rightarrow\mathrm{H}_2(\widehat{X}_0)$ is given by inclusion of topological spaces. The class~$S_{i,j}$ is a generator of $\mathrm{H}_2(\mathbf S^2_{i,j})\simeq\mathbf Z$ and the classes $(A_i)^n_{i=1},(B_j)^m_{j=1}$ are the $n+m$ generators of: 
$$\mathrm{H}_3(X_0)\cong(\mathrm{H}_2(E\I_0)\otimes\mathrm{H}_1(E\II_0))\oplus(\mathrm{H}_1(E\I_0)\otimes\mathrm{H}_2(E\II_0))$$
\end{lem}
%\newpage

\begin{proof}
We first prove statements (a) and (c), then (b).
Each surface $E\J$ is locally of the form $\{xy=z\}$ around a node in $E\J_0$, where $z$ is the local coordinate on a disk around $0$ in $\mathbf P^1$.
We are interested in the degeneration of generic cycles at $z=0$ along some path in the base disk, which we may assume to be the line $\{z>0\}$. Define a continuous function
$$\Phi \; :\; \mathbf R\times\frac{\mathbf R}{\mathbf Z}\times[0,1]\longrightarrow\{xy=z\}\,\subseteq\,\mathbf C^2\times[0,1]$$
$$\Phi(b, t, z) = \left(e^{2\pi it}(-b+\sqrt{b^2+z}),\; e^{-2\pi it}(b+\sqrt{b^2+z}),\; z\right)$$
which commutes with the projection to $z\in[0, 1]$ and let $\Phi_z(b, t)\coloneqq\Phi(b, t, z)$ fiberwise. All~the functions $\Phi_z$, for $z\neq 0$, are homeomorphisms of cylinders. Similarly, $\Phi_0$ is surjective, but only injective on $b\neq 0$; indeed, it sends the cylinder $\mathbf R\times(\mathbf R/\mathbf Z)$ to the cone $\{xy=0\}$ by contracting the circle $\{0\}\times(\mathbf R/\mathbf Z)$ to a point. Moreover, the set $\mathbf R_+\times(\mathbf R/\mathbf Z)$ maps into the line $\{x = 0\}$ and the set $\mathbf R_-\times(\mathbf R/\mathbf Z)$ into the line $\{y = 0\}$.
In our situation, any cycle $\Phi_z(b,-)\subseteq E\J_z$~goes~to the vanishing cycle $\Phi_0(b,-)\subseteq E\J_0$, which is contracted to a point as $b\rightarrow 0$. On the~other~hand, any line $\Phi_z(-,t)\subseteq E\J_z$ can be seen as part of the larger, nondegenerating cycle.
\medskip

Observe a node of $X$ which lies in $0$. Then it is locally of the form $\{x_1y_1 = x_2y_2 = z\}$ and a small resolution by blowing-up the divisor $\{y_1 = y_2 = 0\}$ makes it into the form
$$\big\{\,\big((x_1,y_1),(x_2,y_2),z,(p:q)\big)\in\mathbf C^2\times\mathbf C^2\times\mathbf C\times\mathbf P^1\,\mid\, x_1y_1 = x_2y_2 = z,\; px_2 = qx_1,\; py_1 = qy_2\,\big\}$$
since $x_1/x_2 = y_2/y_1$ on a dense open set. This is clearly the same as taking the small resolution with respect to $\{x_1 = x_2 = 0\}$. By Schoen's construction (specifically Lemma 3.1 in \cite{Sch88}), we always take either this small resolution or the opposite one along $\{x_1 = y_2 = 0\}$ (equivalently $\{y_1 = x_2 = 0\}$). But all resulting formulas in this proof are (anti)symmetric with respect to this choice. This is the independence of choice alluded to in Remark \ref{remcuriosity}.
\medskip

We again consider points $(\Phi_z(b_1,t_1),\Phi_z(b_2,t_2),z)$ of $X$, for $z\in[0,1]$, that lift uniquely~to~some $(\Phi_z(b_1,t_1),\Phi_z(b_2,t_2),z,(u:v))\in\widehat{X}$ when $(b_1,b_2,z)\neq (0,0,0)$. If we make the arbitrary choice $b_1 > 0$, we have $y_1\neq 0$ and $(u : v) = (y_2 : y_1)$, hence we may write this point as:
\vspace{-13pt}

$$\Psi_z(b_1,t_1,b_2,t_2)\coloneqq \left(\Phi_z(b_1,t_1),\Phi_z(b_2,t_2),z,\big(e^{-2\pi it_2}(b_2+\sqrt{\smash{b_2^2}+z}):e^{-2\pi it_1}(b_1+\sqrt{\smash{b_1^2}+z})\big)\right)$$
\vspace{-12pt}

The cycle $uu$ is represented by $\Psi_z(b_1,-,b_2,-)$ for any $b_1,b_2$ in a generic fiber $\smash{\widehat{X}_z}$ over $z\neq 0$. If we let $z = b_2 = 0$, this defines a $2$-cycle in $\widehat{X}_0$ (as it does not intersect the node of $X$, since $b_1 > 0$). However, this $2$-cycle factors through a circle $\Psi_0(b_1,-,0,-)$ (with $x_1 = x_2 = y_2 = 0$, but $|y_1| = 2b_1\neq 0$), and thus its class in $\smash{\widehat{X}}$ is trivial. This proves (a).
\medskip

Now, the node in $X$ is replaced in $\widehat{X}$ by the sphere $\smash{\mathbf S^2_{i_0,j_0}} = \{0,0,0,0,0,(u:v)\mid (u:v)\in\mathbf P^1\}$ whose fundamental class in $\mathrm{H}_2(\widehat{X}_0)$ is $\alpha(S_{i_0,j_0})$. The classes $vu$ and $uv$ are handled analogously, so we only prove the statement for $uv$. We may suppose that in any fiber, near this node, this class is represented by $\Psi_z(b_1,-,-,0)$, which also makes sense as we take $z = 0$. Formally, one may look at this cycle as an element of $\mathrm{Z}_2(\widehat{X}_0, \widehat{X}_0\setminus D)$, where $D$ is some open neighborhood of $\mathbf S^2_{i_0,j_0}$ in $\smash{\widehat{X}_0}$.
It suffices to prove that the difference between this cycle and a cycle representing the element $\pm\alpha(S_{i_0,j_0})$ in $\smash{\mathrm{H}_2(\widehat{X}_0, \widehat{X}_0\setminus D)}$ is a relative boundary.
For this, define a continuous map $F : [0,1]\times(\mathbf R/\mathbf Z)\times\mathbf R\rightarrow\widehat{X}_0$ by letting $F(b_1,t_1,a)$ equal the following piecewise expression:
$$\left\{
\begin{aligned}
    \big(\Phi_0(b_1,t_1),\, &\Phi_0\big(a&&,0\big),0,\big(0 &\hspace{-7pt}:\hspace{-7pt}&& 1&\big)\big)
\;\;\textrm{ if }a\leq 0\\
    \big(\Phi_0(b_1,t_1),\, &\Phi_0\big(ab_1/(1-a)&&,0\big),0,\big(e^{2\pi it_1}a &\hspace{-7pt}:\hspace{-7pt}&& 1-a&\big)\big)
\;\;\textrm{ if }0 < a < 1-\sqrt{b_1}\\
    \big(\Phi_0(b_1,t_1),\, &\Phi_0\big((1-\sqrt{b_1})\sqrt{b_1}&&,0\big),0,\big(e^{2\pi it_1}(1-\sqrt{b_1}) &\hspace{-7pt}:\hspace{-7pt}&& \sqrt{b_1}&\big)\big)
\;\;\textrm{ if }a = 1-\sqrt{b_1}\\
    \big(\Phi_0(b_1,t_1),\, &\Phi_0\big(a-(1-\sqrt{b_1})^2&&,0\big),0,\big(e^{2\pi it_1}\big(a-(1-\sqrt{b_1})^2\big) &\hspace{-7pt}:\hspace{-7pt}&& b_1&\big)\big)
\;\;\textrm{ if }a > 1-\sqrt{b_1}
\end{aligned}
\right.$$
The continuity is clear for $b_1\neq 0$ and easily checked elsewhere. This is the reason that we use $1-\sqrt{b_1}$ everywhere, because the analogous formula with $1-b_1$ is discontinuous at $(0,t_1,1)$.

Also, $F$ is well-defined; it lands into $\widehat{X}_0$. Indeed, if $b_1\neq 0$, then $F(b_1,t_1,a) = \Psi_0(b_1,t_1,b_2,0)$ for some choice of $b_2 = b_2(b_1, a)$. In particular, $F(1,t_1,a) = \Psi_0(1,t_1,a,0)$ for all $t_1,a$, and hence $F(1,-,-)$ is (part of) a representative of $uv$. Otherwise, $b_1 = 0$ and this function becomes:
$$F(0,t_1,a) = \left\{
\begin{aligned}
    \big(0,0,\, &\Phi_0\big(a &&,0\big),0,\big(0 &\hspace{-7pt}:\hspace{-7pt}&& 1&\big)\big)
\;\;\textrm{ if }a < 0\\
    \big(0,0,\, &\Phi_0\big(0 &&,0\big),0,\big(e^{2\pi it_1}a &\hspace{-7pt}:\hspace{-7pt}&& 1-a&\big)\big)
\;\;\textrm{ if }0\leq a\leq 1\\
    \big(0,0,\, &\Phi_0\big(a-1 &&,0\big),0,\big(1 &\hspace{-7pt}:\hspace{-7pt}&& 0&\big)\big)
\;\;\textrm{ if }a > 1
\end{aligned}
\right.$$
The first and third part are just lines, while the middle part parametrizes the sphere $\mathbf S^2_{i_0,j_0}$. Thus we may triangulate the image of $F$ to show that $uv$ and $\pm\alpha(S_{i_0,j})$ agree in $\mathrm{H}_2(\widehat{X}_0, \widehat{X}_0\setminus D)$.  This proves (c), since we can complete the cycle $uv$ by doing the same for $S_{i_0,j}$ with $j\in\{1,\ldots,m\}$ and then summing them all together. The elements $S_{i_0,j}$ are fixed up to sign by their definition; we may choose them such that the expression $\pm\alpha(S_{i_0,j})$ above becomes in fact $-\alpha(S_{i_0,j})$.
\medskip

All that remains now is to prove (b). We start by interpreting the map $\delta$:
The Mayer-Vietoris exact sequence from the proof of Theorem \ref{thmvanish} maps naturally into the long exact sequence in singular homology of the pair $(\widehat{X}_0,\; \bigsqcup\nolimits_{i,j}\mathbf S^2_{i,j})$,
\vspace{-15pt}

\begin{center}\begin{tikzcd}[column sep = small]
    \!\mathrm{H}_k\!\left(\bigsqcup\mathbf S^2_{i,j}\right)\arrow[r]\arrow[d, equal] & \mathrm{H}_k\!\left(\widehat{X}_0\sqcup\bigsqcup\mathbf B^2_{i,j}\right)\arrow[r]\arrow[d] & \mathrm{H}_k\big(X_0\big)\arrow[r, "\delta"]\arrow[d, "\beta"] & \mathrm{H}_{k-1}\!\left(\bigsqcup\nolimits\mathbf S^2_{i,j}\right)\arrow[r]\arrow[d, equal] & \mathrm{H}_{k-1}\!\left(\widehat{X}_0\sqcup\bigsqcup\mathbf B^2_{i,j}\right)\arrow[d]\\
    \!\mathrm{H}_k\!\left(\bigsqcup\mathbf S^2_{i,j}\right)\arrow[r] & \mathrm{H}_k\big(\widehat{X}_0\big)\arrow[r] & \mathrm{H}_k\big(\widehat{X}_0,\;\bigsqcup\nolimits\mathbf S^2_{i,j}\big)\arrow[r, "\delta'"] & \mathrm{H}_{k-1}\!\left(\bigsqcup\mathbf S^2_{i,j}\right)\arrow[r] & \mathrm{H}_{k-1}\big(\widehat{X}_0\big)
\end{tikzcd}\end{center}
\vspace{-5pt}

\noindent
where $\beta : \mathrm{H}_k(X_0)\rightarrow\mathrm{H}_k(\widehat{X}_0, \bigsqcup\nolimits_{i,j}\mathbf S^2_{i,j})$ is constructed by taking a $k$-cycle $\gamma$ in $\widehat{X}_0\cup\bigsqcup\mathbf B^2_{i,j}$ (which is homotopy equivalent to $X_0$) and modifying its part inside $\bigsqcup\mathbf B^2_{i,j}$ so that it becomes a $k$-chain with boundary in $\mathbf S^2_{i,j}$. An alternative and more general way of doing this construction (and also the reason why the Mayer-Vietoris sequence exists in the first place) is by saying that each $\mathbf S^2_{i,j}\subseteq\widehat{X}_0$ is a deformation retract of some neighborhood, which is also not difficult to show.

Finally, for $k\geq 2$ all other vertical maps are isomorphisms, hence so is $\beta$ by the 5-lemma. We are interested in the case $k = 3$, and can thus replace $\delta$ by $\delta' : \mathrm{H}_3(\widehat{X}_0,\,\bigsqcup\mathbf S^2_{i,j})\rightarrow\mathrm{H}_2(\bigsqcup\mathbf S^2_{i,j})$, which is easier to work with. Writing
$$\mathrm{H}_3(\widehat{X}_0,\,\bigsqcup\mathbf S^2_{i,j})\cong\mathrm{H}_3(X_0)\cong(\mathrm{H}_2(E\I_0)\otimes\mathrm{H}_1(E\II_0))\oplus(\mathrm{H}_1(E\I_0)\otimes\mathrm{H}_2(E\II_0))$$
we take as $A'_i\in\mathrm{H}_3(\widehat{X}_0,\,\bigsqcup\mathbf S^2_{i,j})$ the product of the nondegenerate cycle in $\mathrm{H}_1(E\II_0)$ with the $i$-th of the $n$ ``pockets'' in $\mathrm{H}_2(E\I_0)$ of the $n$-pinched torus $E\I_0$. Then $\delta'(A'_i) = \delta(A_i)$ for $A_i\in\mathrm{H}_3(X_0)$. Clearly then $\delta'(A'_{i_0})$ is of the form $\sum_j{k_{i_0,j}S_{i_0,j}}+\sum_j{k_{i_0+1,j}S_{i_0+1,j}}$ for some $k_{i,j}\in\mathbf Z$. However, if we take a small enough neighborhood $D$ of $\mathbf S^2_{i_0,j_0}$, then we can assume that $A'_{i_0}$ is represented by a triangulation of the image of $F$ in the quotient $\mathrm{Z}_3(\widehat{X}_0,\,D\sqcup\mathbf S^2_{i_0,j_0})$ of $\mathrm{Z}_3(\widehat{X}_0,\,\bigsqcup\mathbf S^2_{i,j})$, where $F$ is the map constructed earlier. But we've seen from the form of $F(0,-,-)$ that this implies $k_{i_0,j_0} = 1$. Similarly, $k_{i_0+1,j_0} = -1$ by mirroring $F$. This is true for all $j_0$, which ends the proof of (b) and of the lemma. 
\end{proof}

To conclude, we have found an explicit way to express the vanishing cycles of $X$ and $\widehat{X}$ at a point $s\in\Sigma$. These results supply the final ingredient to the procedure described in Remark~\ref{remsecondmain}, allowing us to calculate the periods of $\widetilde{\omega}$. In the following section, we catalog the results of this computation in our distinguished cases of Type I, II and III.
\section{Final Results}

This chapter contains the final results of this paper, computed explicitly for all examples of Type I, II or III as in Definition \ref{defntypes}. For a $3$-form $\widetilde{\omega}$ as defined in Section 2, we denote by $\mathcal I(\widehat{X})$ and $\mathcal I(X)$ its respective $\mathbf Z$-modules of periods in $\mathbf C$ (see Subsection 7.3 for the modified $3$-forms $z^n\widetilde{\omega}$ considered in Type III). In each such example we have applied our procedure, summed up in Example \ref{exmpmain} and Remark $\ref{remsecondmain}$, to compute a basis of $\mathcal I(\widehat{X})$, resp.\ $\mathcal I(X)$. These bases~are~listed below, along with the equations $g_2\J,g_3\J$ defining the elliptic surfaces $E\J$ in each example.

\begin{rem}\label{remfinal}
As discussed in Section 2, the examples of Type I and II are rigid and $\mathrm{rk}_{\mathbf Z}\,\mathcal I(\widehat{X}) = 2$. In general, $\mathcal I(\widehat{X})\subseteq\mathcal I(X)$ and the latter group can have bigger rank (such as $3,4,5$ in Type II, see Subsection 7.2).
On the other hand, it is easy to see that $\mathrm{rk}_{\mathbf Z}\,\mathcal I(X) = \mathrm{rk}_{\mathbf Z}\,\mathcal I(\widehat{X})$ whenever $E\I = E\II$ (so in particular in our Types I and III, but not II): In the notation of Section 6,
$$VC'_s(X) = (\mathbf Z\cdot uu)\oplus(\mathbf Z\cdot uv)\oplus(\mathbf Z\cdot vu)
\textrm{ ,}\;\;\;\;
VC'_s(\widehat{X}) = (\mathbf Z\cdot uu)\oplus (n\mathbf Z\cdot (uv + vu))$$
for a point $s\in\mathbf P^1$ with fiber of type $I_n$. Given $\ell\in VC'_s(X)$, we have $n(\ell+\bar{\ell})\in VC_s(\widehat{X})$, where $\bar{\ell}$ is the cycle symmetric to $\ell$ (with respect to $uv$ and $vu$). Since integration of $\widehat{\omega}$ is preserved under this symmetry of $E\I$ and $E\II$, we conclude from Theorem \ref{thmperMV} that $2d\cdot\mathcal I(X)\subseteq\mathcal I(\widehat{X})$~for:
$$\smash{d\coloneqq \mathrm{lcm}\left\{n\in\mathbf Z\mid (\exists\, s\in\mathbf P^1)\, E_s\I\textrm{ is of type }I_n\right\}}$$
Hence $\mathcal I(X)$ and $\mathcal I(\widehat{X})$ indeed have the same rank $r$ and the index $[\mathcal I(X) : \mathcal I(\widehat{X})]$ divides $(2d)^r$. This observation can be seen to hold in all examples of Type I and III.
\end{rem}

Finally, each example admits an associated modular form, which can be found in \cite{Mey05}.~The notation $N/m$ designates ``the $m$-th newform of weight $4$ for $\Gamma_0(N)$ with rational coefficients'' (see \cite[1.8.3]{Mey05} for more details), also used in the algebra system Magma. For each example, we provide this modular form and its two periods, which can be related to our results by factors that are quotients of small integers. %\! These relations confirm the precision of~our~results.\;\smallskip
In~particular, one can confirm the numerical precision~of~our results by checking that these relations indeed hold.

All of the values printed below, as well as those excluded for practical reasons, can be found~in plain-text tables (also readable by Maple) at the following link:\\
\centerline{\url{https://github.com/donlagic-azur/numcalc-schoen-supplement/tree/main/results}}

\subsection{Type I}
We take $E\I=E\II$ to be an elliptic surface with exactly $4$ (semistable) singular fibers. The list of (all $6$) such surfaces can be found in \cite[\textsection 4, Table 1]{Sch88}. Only $4$ of them have all singular fibers lying above real points, and they are determined by the Kodaira~types of their singular fibers: $(I_2,I_2,I_4,I_4),\;(I_1,I_1,I_2,I_8),\;(I_1,I_1,I_5,I_5),\;(I_1,I_2,I_3,I_6)$. Explicit forms of these surfaces can be found in \cite{Her91}.
As discussed above, $\mathrm{rk}_{\mathbf Z}\,\mathcal I(X) = \mathrm{rk}_{\mathbf Z}\,\mathcal I(\widehat{X}) = 2$ in all examples. The first two examples are isogenous, hence give the same results. %\bigskip\bigskip
\vfill

\hrule
\medskip
\noindent
The surface $E\I=E\II$ with singular fibers of types $(I_1,I_1,I_2,I_8)$ is defined over $\mathbf P^1$ by functions
\begin{align*}
g_2(X,Y)&\coloneqq 3(16X^4 - 16X^2Y^2 + Y^4)\\
g_3(X,Y)&\coloneqq 64X^6 - 96X^4Y^2 + 30X^2Y^4 + Y^6
\end{align*}
and these fibers lie, respectively, over points: $-1,1,0,\infty$
$$\mathcal I(\widehat{X})\textrm{ has basis }
\left\{\begin{aligned}
18.6601680444816862921513542049 &+ 0 \, i \\
0 &+ 23.1231662167091830576134155719 \, i 
\end{aligned}\right\}$$
$$\mathcal I(X)\textrm{ has basis }
\left\{\begin{aligned}
9.33008402224084314607567710247 &+ 0\, i \\
0 &+ 5.78079155417729576440335389298 \, i 
\end{aligned}\right\}$$
The associated modular form $16/1$ ($16.4.a.a$ in \cite{LMFDB}) has periods $\omega_1,\omega_2$ such that
$$\begin{aligned}
&\pi^2\omega_1 = & 1.08389841640824295582562885492 &+ 0\, i \\
&\pi^2\omega_2 = & 0 &- 1.74939075417015808988918945670 \, i 
\end{aligned}$$
and the above basis of $\smash{\mathcal I(\widehat{X})}$ can be written as: $\smash{(64i/6)\pi^2\omega_2}$ and $\smash{(64i/3)\pi^2\omega_1}$
\pagebreak

\hrule
\medskip
\noindent
The surface $E\I=E\II$ with singular fibers of types $(I_2,I_2,I_4,I_4)$ is defined over $\mathbf P^1$ by functions
\begin{align*}
g_2(X,Y)&\coloneqq 12(X^4 - X^2Y^2 + Y^4)\\
g_3(X,Y)&\coloneqq 4(2X^6 - 3X^4Y^2 - 3X^2Y^4 + 2Y^6)
\end{align*}
and these fibers lie, respectively, over points: $-1,1,0,\infty$
$$\mathcal I(\widehat{X})\textrm{ has basis }
\left\{\begin{aligned}
18.6601680444816862921513542049 &+ 0 \, i \\
0 &+ 23.1231662167091830576134155719 \, i 
\end{aligned}\right\}$$
$$\mathcal I(X)\textrm{ has basis }
\left\{\begin{aligned}
9.33008402224084314607567710247 &+ 0\, i \\
0 &+ 5.78079155417729576440335389298 \, i 
\end{aligned}\right\}$$
The associated modular form $8/1$ ($8.4.a.a$ in \cite{LMFDB}) has periods $\omega_1,\omega_2$ such that
$$\begin{aligned}
&\pi^2\omega_1 = & 0 &+ 1.08389841640824295582562885492\, i \\
&\pi^2\omega_2 = & 1.74939075417015808988918945670 &+ 0 \, i 
\end{aligned}$$
and the above basis of $\smash{\mathcal I(\widehat{X})}$ can be written as: $\smash{(64/6)\pi^2\omega_2}$ and $\smash{(64/3)\pi^2\omega_1}$
\vfill
%\bigskip

\hrule
\medskip
\noindent
The surface $E\I=E\II$ with singular fibers of types $(I_1,I_1,I_5,I_5)$ is defined over $\mathbf P^1$ by functions
\begin{align*}
g_2(X,Y)&\coloneqq 3(X^4 - 12X^3Y + 14X^2Y^2 + 12XY^3 + Y^4)\\
g_3(X,Y)&\coloneqq X^6 - 18X^5Y + 75X^4Y^2 + 75X^2Y^4 + 18XY^5 + Y^6
\end{align*}
and these fibers lie, respectively, over points: $\left(\frac{1-\sqrt{5}}{2}\right)^5,\left(\frac{1+\sqrt{5}}{2}\right)^5,0,\infty$
$$\mathcal I(\widehat{X})\textrm{ has basis }
\left\{\begin{aligned}
13.8030652044679021961193422703 &+ 0 \, i\\
0 &+ 21.5650087259302781487215564481 \, i
\end{aligned}\right\}$$
$$\mathcal I(X)\textrm{ has basis }
\left\{\begin{aligned}
6.90153260223395109805967113517 &+ 0 \, i\\
0 &+ 2.15650087259302781487215564481 \, i 
\end{aligned}\right\}$$
The associated modular form $5/1$ ($5.4.a.a$ in \cite{LMFDB}) has periods $\omega_1,\omega_2$ such that
$$\begin{aligned}
&\pi^2\omega_1 = & 2.07045978067018532941790134055 &- 1.94085078533372503338494008033\, i\\
&\pi^2\omega_2 = & -2.07045978067018532941790134055 &+ 2.58780104711163337784658677377\, i
\end{aligned}$$
and the above basis of $\smash{\mathcal I(\widehat{X})}$ can be written as: $\smash{(80/3)\pi^2\omega_1+20\pi^2\omega_2}$ and $\smash{(100/3)\pi^2(\omega_1+\omega_2)}$
\vfill
%\bigskip

\hrule
\medskip
\noindent
The surface $E\I=E\II$ with singular fibers of types $(I_1,I_2,I_3,I_6)$ is defined over $\mathbf P^1$ by functions
\begin{align*}
g_2(X,Y)&\coloneqq 12(X^4 - 4X^3Y + 2XY^3 + Y^4)\\
g_3(X,Y)&\coloneqq 4(2X^6 - 12X^5Y + 12X^4Y^2 + 14X^3Y^3 + 3X^2Y^4 + 6XY^5 + 2Y^6)
\end{align*}
and these fibers lie, respectively, over points: $4,-1/2,0,\infty$
$$\mathcal I(\widehat{X})\textrm{ has basis }
\left\{\begin{aligned}
7.29113457566789259514571215945 &+ 19.2156302580844571356028493573 \, i \\
7.29113457566789259514571215945 &- 19.2156302580844571356028493573 \, i 
\end{aligned}\right\}$$
$$\mathcal I(X)\textrm{ has basis }
\left\{\begin{aligned}
7.29113457566789259514571215945 &+ 0\, i \\
0 &+ 3.20260504301407618926714155955 \, i 
\end{aligned}\right\}$$
The associated modular form $6/1$ ($6.4.a.a$ in \cite{LMFDB}) has periods $\omega_1,\omega_2$ such that
$$\begin{aligned}
&\pi^2\omega_1 = & 0 &+ 0.800651260753519047316785389880\, i\\
&\pi^2\omega_2 = & -1.82278364391697314878642803987 &+ 4.00325630376759523658392694940\, i
\end{aligned}$$
and the above basis of $\smash{\mathcal I(\widehat{X})}$ can be written as: $\smash{44\pi^2\omega_1-4\pi^2\omega_2}$ and $\smash{-4\pi^2\omega_1-4\pi^2\omega_2}$
%\medskip\medskip\medskip
\pagebreak

\subsection{Type II}
We take $E\I$ to be the surface with singular fibers of types $(I_1,I_2,I_3,I_6)$, defined by $g_2\I,g_3\I$ as above. If $\phi$ is a M\"obius transformation permuting $-1/2,0,\infty$ (and hence the singular fibers $I_2,I_3,I_6$), we choose $E\II$ to be the pullback of $E\I$ by $\phi$. This determines the functions $g_2\II,g_3\II$ up to factors $u^4,u^6$, and we write in each example the particular choice for which we calculated the results.

As expected, the lattice $\mathcal I(\widehat{X})$ has rank $2$. However, the groups $\mathcal I(X)$ have higher rank. Out of the five examples (there are five nontrivial elements $\phi$ in the permutation group $\Sigma_3$), we note that, in the order below: two have rank $3$, one has rank $4$, two have rank $5$.

The two cases with rank $5$ are the ones given by automorphisms of order $3$ in $\Sigma_3$ (the rest are given by automorphisms of order $2$). Their results are related by a factor of $4$, which is expected since the two threefolds are isogenous (which can be seen after exchanging $E\I $~and~$E\II$), and they have the same associated modular form. \vfill

\hrule
\medskip
\noindent
The pullback $E\II$ of $E\I$ by $z\mapsto -z/(2z+1)$ is, up to rescaling, given by functions
\begin{align*}
    g_2\II(z,1) &= 108z^4 + 144z^3 + 144z^2 + 72z + 12\\
    g_3\II(z,1) &= -216z^6 - 432z^5 + 360z^3 + 252z^2 + 72z + 8
\end{align*}
and its fibers of types $(I_1,I_2,I_3,I_6)$ lie, respectively, over points: $-4/9,\infty,0,-1/2$

$$\mathcal I(\widehat{X})\textrm{ has basis }
\left\{\begin{aligned}
10.1775023004434042782842363056 &+ 0 \, i\\
0 &+ 18.4818859696295144823284523403 \, i
\end{aligned}\right\}$$
$$\mathcal I(X)\textrm{ has basis }
\left\{\begin{aligned}
2.91210278431660900940154286414 &+ 0 \, i\\
10.1775023004434042782842363056 &+ 0 \, i\\
0 &+ 1.84818859696295144823284523403 \, i
\end{aligned}\right\}$$
The associated modular form $10/1$ ($10.4.a.a$ in \cite{LMFDB}) has periods $\omega_1,\omega_2$ such that
$$\begin{aligned}
&\pi^2\omega_1 = & 1.52662534506651064174263544586 &+ 4.15842434316664075852390177654\, i\\
&\pi^2\omega_2 = & 0 &- 1.38614144772221358617463392551\, i
\end{aligned}$$
and the above basis of $\smash{\mathcal I(\widehat{X})}$ can be written as: $\smash{(20/3)\pi^2\omega_1+20\pi^2\omega_2}$ and $\smash{-(40/3)\pi^2\omega_2}$
\medskip\bigskip%\pagebreak

\hrule
\medskip
\noindent
The pullback $E\II$ of $E\I$ by $z\mapsto 1/(4z)$ is, up to rescaling, given by functions
\begin{align*}
    g_2\II(z,1) &= \dfrac{3}{64} - \dfrac{3}{4}z + 6z^3 + 12z^4\\
    g_3\II(z,1) &= \dfrac{1}{512} - \dfrac{3}{64}z + \dfrac{3}{16}z^2 + \dfrac{7}{8}z^3 + \dfrac{3}{4}z^4 + 6z^5 + 8z^6
\end{align*}
and its fibers of types $(I_1,I_2,I_3,I_6)$ lie, respectively, over points: $1/16,-1/2,\infty,0$

$$\mathcal I(\widehat{X})\textrm{ has basis }
\left\{\begin{aligned}
7.02101068370533129948622785532 &+ 56.5585057250058555647892632590 \, i\\
7.02101068370533129948622785532 &- 56.5585057250058555647892632590 \, i
\end{aligned}\right\}$$
$$\mathcal I(X)\textrm{ has basis }
\left\{\begin{aligned}
7.02101068370533129948622785532 &+ 1.27530934398204965297269158496 \, i\\
7.02101068370533129948622785532 &- 1.27530934398204965297269158496 \, i\\
0 &+ 5.85917529922885628954124316713 \, i
\end{aligned}\right\}$$
The associated modular form $21/2$ ($21.4.a.a$ in \cite{LMFDB}) has periods $\omega_1,\omega_2$ such that
$$\begin{aligned}
&\pi^2\omega_1 = & -0.752251144682714067802095841641 &- 2.01994663303592341302818797353\, i\\
&\pi^2\omega_2 = & 2.75825419716995158194101808602 &+ 6.05983989910777023908456392067\, i
\end{aligned}$$
and the above basis of $\smash{\mathcal I(\widehat{X})}$ can be written as: $\smash{-112\pi^2\omega_1-28\pi^2\omega_2}$ and $\smash{196\pi^2\omega_1+56\pi^2\omega_2}$
\medskip\bigskip\pagebreak

\hrule
\medskip
\noindent
The pullback $E\II$ of $E\I$ by $z\mapsto -1/2-z$ is, up to rescaling, given by functions
\begin{align*}
    g_2\II(z,1) &= \dfrac{27}{4} + 18z + 90z^2 + 72z^3 + 12z^4\\
    g_3\II(z,1) &= -\dfrac{27}{8} - \dfrac{27}{2}z + \dfrac{135}{2}z^2 + 180z^3 + 198z^4 + 72z^5 + 8z^6
\end{align*}
and its fibers of types $(I_1,I_2,I_3,I_6)$ lie, respectively, over points: $-9/2,0,-1/2,\infty$

$$\mathcal I(\widehat{X})\textrm{ has basis }
\left\{\begin{aligned}
15.6841094746081896506771079054 &+ 0 \, i\\
0 &+ 17.0625437087419634029794240275 \, i
\end{aligned}\right\}$$
$$\mathcal I(X)\textrm{ has basis }
\left\{\begin{aligned}
15.6841094746081896506771079054 &+ 0 \, i\\
3.25863063246525860751003087346 &+ 8.53127185437098170148971201373 \, i\\
3.25863063246525860751003087346 &- 8.53127185437098170148971201373 \, i\\
0 &+ 2.27403622608502294651981388337 \, i
\end{aligned}\right\}$$
The associated modular form $17/1$ ($17.4.a.a$ in \cite{LMFDB}) has periods $\omega_1,\omega_2$ such that
$$\begin{aligned}
&\pi^2\omega_1 = & 1.38389201246542849858915657984 &+ 4.51655568760816678314161224223\, i\\
&\pi^2\omega_2 = & -2.76778402493085699717831315967 &- 10.5386299377523891606637618985\, i
\end{aligned}$$
and the above basis of $\smash{\mathcal I(\widehat{X})}$ is given by:\! $\smash{(238/3)\pi^2\omega_1+34\pi^2\omega_2}$ and $\smash{-(68/3)\pi^2\omega_1-(34/3)\pi^2\omega_2}$
\!\!\medskip\bigskip

\hrule
\medskip
\noindent
The pullback $E\II$ of $E\I$ by $z\mapsto -1/(4z+2)$ is, up to rescaling, given by functions
\begin{align*}
    g_2\II(z,1) &= 3072z^4 + 4608z^3 + 2304z^2 + 576z + 108\\
    g_3\II(z,1) &= 32768z^6 + 73728z^5 + 64512z^4 + 23040z^3 - 1728z - 216
\end{align*}
and its fibers of types $(I_1,I_2,I_3,I_6)$ lie, respectively, over points: $-9/16,0,\infty,-1/2$

$$\mathcal I(\widehat{X})\textrm{ has basis }
\left\{\begin{aligned}
11.4118420461744051807307248949 &+ 0 \, i\\
0 &+ 4.21820528193890201592231653461 \, i
\end{aligned}\right\}$$
$$\mathcal I(X)\textrm{ has basis }
\left\{\begin{aligned}
1.15964224207708731147527428542 &+ 0 \, i\\
1.72900723259046311316331046608 &+ 0 \, i\\
5.70592102308720259036536244744 &+ 1.79835796276720374170952048120 \, i\\
5.70592102308720259036536244744 &- 1.79835796276720374170952048120 \, i\\
0 &+ 4.21820528193890201592231653461 \, i
\end{aligned}\right\}$$
The associated modular form $73/1$ ($73.4.a.a$ in \cite{LMFDB}) has periods $\omega_1,\omega_2$ such that
$$\begin{aligned}
&\pi^2\omega_1 = & 8.91061639221837116851577154289 &+ 0.520052705992467371826039094834\, i\\
&\pi^2\omega_2 = & -2.81387886070053826374182256461 &- 0.173350901997489123942012962848\, i
\end{aligned}$$
and the above basis of $\smash{\mathcal I(\widehat{X})}$ is given by:\! $\smash{(73/3)\pi^2\omega_1+73\pi^2\omega_2}$ and $\smash{-146\pi^2\omega_1-(1387/3)\pi^2\omega_2}$
\!\!\medskip\bigskip%\pagebreak

\hrule
\medskip
\noindent
The pullback $E\II$ of $E\I$ by $z\mapsto -(2z+1)/(4z)$ is, up to rescaling, given by functions
\begin{align*}
    g_2\II(z,1) &= \dfrac{27}{4}z^4 + \dfrac{9}{2}z^3 + \dfrac{45}{8}z^2 + \dfrac{9}{8}z + \dfrac{3}{64}\\
    g_3\II(z,1) &= -\dfrac{27}{8}z^6 - \dfrac{27}{8}z^5 + \dfrac{135}{32}z^4 + \dfrac{45}{16}z^3 + \dfrac{99}{128}z^2 + \dfrac{9}{128}z + \dfrac{1}{512}
\end{align*}
and its fibers of types $(I_1,I_2,I_3,I_6)$ lie, respectively, over points: $-1/18,\infty,-1/2,0$

$$\mathcal I(\widehat{X})\textrm{ has basis }
\left\{\begin{aligned}
45.6473681846976207229228995794 &+ 0 \, i\\
0 &+ 16.8728211277556080636892661384 \, i
\end{aligned}\right\}$$
$$\mathcal I(X)\textrm{ has basis }
\left\{\begin{aligned}
4.63856896830834924590109714169 &+ 0 \, i\\
6.91602893036185245265324186434 &+ 0 \, i\\
22.8236840923488103614614497897 &+ 7.19343185106881496683808192482 \, i\\
22.8236840923488103614614497897 &- 7.19343185106881496683808192482 \, i\\
0 &+ 16.8728211277556080636892661384 \, i
\end{aligned}\right\}$$
The associated modular form $73/1$ ($73.4.a.a$ in \cite{LMFDB}) has periods $\omega_1,\omega_2$ such that
$$\begin{aligned}
&\pi^2\omega_1 = & 8.91061639221837116851577154289 &+ 0.520052705992467371826039094834\, i\\
&\pi^2\omega_2 = & -2.81387886070053826374182256461 &- 0.173350901997489123942012962848\, i
\end{aligned}$$
and the above basis of $\smash{\mathcal I(\widehat{X})}$ is given by:\! $\smash{(73/3)\pi^2\omega_1+73\pi^2\omega_2}$ and $\smash{-146\pi^2\omega_1-(1387/3)\pi^2\omega_2}$
\!\!\medskip\bigskip

\subsection{Type III}

To find formulas for the modular surface $E\I = E\II$ over $X_1(N), N\in\{7,8,10\}$, we follow \cite{Baa10} (one can also find a discussion of Tate normal forms in \cite[Chapter~4,~\textsection4]{Hus04}): We may parameterize by points $(f,g)\in\mathbf A^2$ the elliptic curves of the form
$$y^2 + ((1 + g)x + f)y = x^3 + fx^2$$
and then the condition of the point $(0,0)$ being an $N$-torsion point is a polynomial condition in $f,g$. In other words, the curve $X_1(N)$ parametrizing elliptic curves with a fixed $N$-torsion point is exactly the curve defined in $\mathbf A^2$ by a single polynomial equation $\Phi_N(f,g) = 0$. The~following equations are calculated in \cite{Baa10}:
\begin{align*}
\Phi_7(f,g) &= g^3-fg+f^2\\
\Phi_8(f,g) &= (f+1)g^2-3fg+2f^2\\
\Phi_{10}(f,g) &= g^5+fg^4-3fg^3+(3f^2+f)g^2-2f^2g+f^3
\end{align*}
We know that the resulting curves are birationally isomorphic to $\mathbf P^1$, so we may parametrize them in Maple to find functions $f(z),g(z)$, for the local coordinate $z$ on $\mathbf P^1$. After converting the resulting fibers to Weierstrass form (and clearing denominators), we may then determine the functions $g_2,g_3$ as previously, which will be presented below.
\smallskip

On $\widehat{X}$ and $X$, we consider a $3$-form of the form $\sum_{n=0}^{2k-2}c_nz^n\widetilde{\omega}$ in the notation of Section 2. Here, $k=2$ for $N\in\{7,8\}$ and $k=3$ for $N=10$. For each $N$, we will present $2k-1$ vectors $\widehat{R}_n(N)$, resp.\ $R_n(N)$, such that the module of periods $\mathcal I(\widehat{X})$ (of the chosen $3$-form~$\sum_{n=0}^{2k-2}c_nz^n\widetilde{\omega}$), resp.\ $\mathcal I(X)$, is spanned by the components of the vector $\sum_{n=0}^{2k-2}c_n\widehat{R}_n(N)$, resp.\ $\sum_{n=0}^{2k-2}c_nR_n(N)$. For a generic choice of scalars $c_n$, this spanning set is a basis (consisting of $4k-2$ elements). We also record the generic value of the index $[\mathcal I(X) : \mathcal I(\widehat{X})]$ (see Remark \ref{remfinal}).
\medskip

Finally, note that there are many special choices of $c_n$ for which this spanning set is not~a~basis. Some examples can be found (but searching for them systematically is in general very difficult):
\begin{itemize}
    \item For $N = 8$, setting $c_0=c_2=0$ gives rank $2$ (as opposed to the generic rank value of $6$) and setting $c_1=0$ gives rank $4$ for generic choices of $c_0,c_2$, while it gives rank $2$ when e.g. $-c_2/c_0\in\big\{\,4+4\sqrt{2},\; 12+8\sqrt{2},\; 24+16\sqrt{2}\,\big\}$ -- see the examples below 
    %\item For $N = 8$, setting $c_1=0$ and $-c_2/c_0\in\big\{\,4+4\sqrt{2},\; 12+8\sqrt{2},\; 24+16\sqrt{2}\,\big\}$ gives rank $2$\;
    \item For $N = 10$, setting $c_n = 0$ for all $n\neq n_0$ gives modules of rank $8$ (as opposed to $10$)
    \item For $N = 10$, setting $c_0 = c_2\neq 0$ and $c_1 = c_3 = c_4 = 0$ gives modules of rank $6$
\end{itemize}
\medskip

\hrule
\medskip
\noindent
The modular surface $E\I=E\II$ over $X_1(7)$ is defined over $\mathbf P^1\simeq X_1(7)$ by the following~functions
\!\!\!\vspace{-15pt}

\begin{align*}
g_2(z,1)&= \frac{1}{12}(z^2 - z + 1)(z^6 - 11z^5 + 30z^4 - 15z^3 - 10z^2 + 5z + 1)\\
g_3(z,1)&= \frac{1}{216}(z^{12} - 18z^{11} + 117z^{10} - 354z^9 + 570z^8 - 486z^7 +\\
&\;\;\;\;\;\;\;\;\;\;\;\;\;\;\;\;\;\;\;\;\;\;\;+ 273z^6 - 222z^5 + 174z^4 - 46z^3 - 15z^2 + 6z + 1)
\end{align*}
\vspace{-15pt}

\noindent
and it has $6$ singular fibers of types $(I_1,I_1,I_1,I_7,I_7,I_7)$ which lie, respectively, over the points:
$$\{\textrm{the three roots of }z^3 - 8z^2 + 5z + 1\},0,1,\infty$$
The generic value of $[\mathcal I(X) : \mathcal I(\widehat{X})]$ is $784$ and these modules are determined by the $3$ vectors $\widehat{R}_n(7)$, resp.\ $R_n(7)$, at: 
%\smallskip
\url{https://github.com/donlagic-azur/numcalc-schoen-supplement/tree/main/results/Results_Type_III_N(7).mpl}
\bigskip

\hrule
\medskip
\noindent
The modular surface $E\I=E\II$ over $X_1(8)$ is defined over $\mathbf P^1\simeq X_1(8)$ by the following functions (in which we have applied a harmless shift $z\mapsto z+1/2$ to simplify some computations)
\vspace{-15pt}

\begin{align*}
g_2(z,1)&= 768z^8 + 5376z^6 - 480z^4 - 48z^2 + 3\\
g_3(z,1)&= (16z^4 + 8z^2 - 1)(256z^8 - 4352z^6 + 608z^4 - 16z^2 + 1)
\end{align*}
\vspace{-15pt}

\noindent
and it has $6$ singular fibers of types $(I_1,I_1,I_2,I_4,I_8,I_8)$ which lie, respectively, over the points:
$$-\frac{\sqrt{2}}{4},\frac{\sqrt{2}}{4},\infty,0,\frac{-1}{2},\frac{1}{2}$$
The generic value of $[\mathcal I(X) : \mathcal I(\widehat{X})]$ is $512$ and these modules are determined by the $3$ vectors $\widehat{R}_n(8)$, resp.\ $R_n(8)$, at: 
%\smallskip
\url{https://github.com/donlagic-azur/numcalc-schoen-supplement/tree/main/results/Results_Type_III_N(8).mpl}
\medskip\medskip%\medskip

One can observe that the components of the vector $\widehat{R}_1(8)$ span a rank $2$ module with basis
$$\left\{\begin{aligned}
(4/3)&\pi^2\omega_2 = & 2.33252100556021078651891927562 &+ 0\, i\\
(8/3)&\pi^2\omega_1 = & 0 &+ 2.89039577708864788220167694649\, i
\end{aligned}\right\}$$
where $\omega_1,\omega_2$ are periods of the associated modular form $8/1$ ($8.4.a.a$ in \cite{LMFDB}), such that:
$$\begin{aligned}
&\pi^2\omega_1 = & 0 &+ 1.08389841640824295582562885492\, i \\
&\pi^2\omega_2 = & 1.74939075417015808988918945670 &+ 0 \, i 
\end{aligned}$$
On the other hand, the vectors $\widehat{R}_0(8)$ and $\widehat{R}_2(8)$ both span rank $4$ modules with bases:
$$\textrm{for }\widehat{R}_0(8),\left\{\begin{aligned}
    21.9094368368078925071642210864 &+ 16.6922982676850234159112792153\, i\\
    38.7307783982130059726452261364 &- 2.86394560838090965600666641064\, i\\
    12.1697358539854246980008039659 &+ 27.9023930099084447798735852337\, i\\
    -24.5674972410477146600100502773 &+ 2.86394560838090965600666641064\, i\\
\end{aligned}\right\}$$
$$\textrm{for }\widehat{R}_2(8),\left\{\begin{aligned}
    0.469883134160358495657537680113 &+ 0.715986402095227414001666602660\, i\\
    4.01070342345168132381633164488 &+ 1.72854408241301421998807660059\, i\\
    -3.23706593039595321705412060909 &- 0.654564479270173098985576696599\, i\\
    -4.34296097398165909292206029375 &- 1.72854408241301421998807660059\, i\\
\end{aligned}\right\}$$
This is explained by the modular form $8/2$ ($8.4.b.a$ in \cite{LMFDB}) with coefficients~in~$\mathbf Q(\sqrt{-7}\,)$. Its periods naturally form a lattice in $\mathbf C^2$ spanned by the following pairs, in which all numbers can be expressed using the coefficients of $\widehat{R}_0(8)$ and $\widehat{R}_2(8)$:
$$\left(\begin{aligned}
    1.87780352814060548667506909681 &-2.21310699350035127773721115166\, i\\
    0 &- 0.189854565059585026122451874612\, i
\end{aligned}\right)$$
$$\left(\begin{aligned}
    1.52539117752033661493191583674 &- 2.97252525373869138222701865010\, i\\
    -1.11510793938043717920911117845 &+ 0.726844366631005586623701826606\, i
\end{aligned}\right)$$
$$\left(\begin{aligned}
    5.98582293504208533176836055050 &+ 4.04650485688153250322951855411\, i\\
    0.762695588760168307465957918366 &+ 0.838058595964465358617980599834\, i
\end{aligned}\right)$$
$$\left(\begin{aligned}
    -0.352412350620268871743153260073 &- 3.69936962036969696885072047671\, i\\
    -1.02700485172536996127332286343 &+ 0.173567618255917767189399038691\, i
\end{aligned}\right)$$

Lastly, we note that the $\mathbf Z$-modules spanned by $\widehat{R}_0(8)$ and $\widehat{R}_2(8)$ are mutually commensurable. Moreover, some combinations $c_0\widehat{R}_0(8)+c_2\widehat{R}_2(8)$ (where one may also use the $4$-term~bases~listed above) span lattices of rank $2$, and we list three cases in which they are pairwise commensurable:
$$(c_0,c_2)\in\left\{\;\big(1,\;-12-8\sqrt{2}\big),\;\big(3+\sqrt{2},\;-20-16\sqrt{2}\big),\;\big(4-\sqrt{2},\;-64-40\sqrt{2}\big)\;\right\}$$
\smallskip

\hrule
\medskip
\noindent
The modular surface $E\I=E\II$ over $X_1(10)$ is defined over $\mathbf P^1\simeq X_1(10)$ by the two functions
\vspace{-15pt}

\begin{align*}
g_2(z,1)&= 3z^{12} - 48z^{11} + 312z^{10} - 1080z^9 + 2160z^8 - 1728z^7 - 3072z^6 +\\
&\;\;\;\;\;\;\;\;\;\;\;\;\;\;\;\;\;\;\;\;\;\;\;+ 10368z^5 - 11520z^4 + 3840z^3 + 3072z^2 - 3072z + 768\\
g_3(z,1)&= -(z^2 - 2z + 2)(z^4 - 8z + 8)(z^4 - 6z^3 + 6z^2 + 4z - 4)\\
&\;\;\;\;\;\;\;\;\;\;\;\;\;\;\;\;\;\;(z^8 - 16z^7 + 104z^6 - 352z^5 + 584z^4 - 384z^3 - 64z^2 + 192z - 64)
\end{align*}
\vspace{-15pt}

\noindent
and it has $8$ singular fibers of types $(I_1,I_1,I_2,I_2,I_5,I_5,I_{10},I_{10})$ which lie respectively over~points:
$$\frac{1-\sqrt{5}}{2},\frac{1+\sqrt{5}}{2},3-\sqrt{5},3+\sqrt{5},1,\infty,0,2$$
The presentation of periods in this final example with our standard precision of 30-35 digits pushes the limits of our ability to detect reasonable integer relations between results, such as using the LLL algorithm to determine the rank of modules and rational dependencies between elements. Going further without increasing the number of decimal places, false relations may start appearing, often difficult to distinguish from the correct ones (as among 10 numbers with 35 decimals of precision it becomes easy to always find relations with integer coefficients having 7 or 8 digits). Any larger examples should be computed with more precision.
\smallskip

\noindent
The generic value of $[\mathcal I(X) : \mathcal I(\widehat{X})]$ is $32000$ and these modules are determined by the $5$ vectors $\widehat{R}_n(10)$, resp.\ $R_n(10)$, at:\ 
%\smallskip
\url{https://github.com/donlagic-azur/numcalc-schoen-supplement/tree/main/results/Results_Type_III_N(10).mpl}
\medskip
%\section*{Bibliography}
\linespread{0.9}


\begin{thebibliography}{[Abc12]}
    \setlength{\itemsep}{5 pt}

    %\bibitem[BLR90]{Neron} S.\ Bosch, W.\ L\"utkebohmert and M.\ Raynaud, \textit{N\'eron models} 

    %\bibitem[EGAIV]{EGAIV} A.\ Grothendieck, \textit{El\'ements de G\'eom\'etrie Alg\'ebrique.\ IV.\ \'Etude locale des sch\'emas et des morphismes de sch\'emas}, Publ.\ Math.\ Inst.\ Hautes \'Etudes Sci.\ \textbf{20}, \textbf{24}, \textbf{28}, \textbf{32} (1964-67).\
    
    \bibitem[AHV18]{AHV18} J.\ M.\ Aroca, H.\ Hironaka, J.\ L.\ Vicente, \textit{Complex Analytic Desingularization} (Springer Tokyo, 2018).\
    
    \bibitem[Baa10]{Baa10} H.\ Baaziz, \textit{Equations for the modular curve $X_1(N)$ and models of elliptic curves with torsion points}, Mathematics of Computation, vol.\ \textbf{79}, no.\ \textbf{272} (2010), 2371-2386

    \bibitem[Bea82]{Bea82} A.\ Beauville, \textit{Les familles stables de courbes elliptiques sur $\mathbf P^1$ admettant 4 fibres singuli\`eres}, C.\ R.\ Math.\ Acad.\ Sci.\ Paris \textbf{294} (1982), 657–60

    %\bibitem[Beu09]{Beu09} F.\ Beukers, \textit{Notes on differential equations and hypergeometric functions}, unpublished (2009), link: \url{https://pages.\uoregon.\edu/njp/beukers.\pdf}

    %\bibitem[Beu17]{Beu17} F.\ Beukers, \textit{Hypergeometric Functions, from Riemann till Present}.\ Uniformization, Riemann-Hilbert correspondence, Calabi--Yau manifolds \& Picard--Fuchs equations, 1–19, Adv.\ Lect.\ Math.\ (ALM), \textbf{42}, Int.\ Press, Somerville, MA, 2018.\

    \bibitem[BKSZ]{BKSZ} K.\ B\"onisch, A.\ Klemm, E.\ Scheidegger, D.\ Zagier, \textit{D-brane masses at special fibres of hypergeometric families of Calabi--Yau threefolds, modular forms, and periods}, Commun.\ Math.\ Phys., vol.\ \textbf{405}, art.\ no 134 (2024).\
    
    %\bibitem[BHPV]{BHPV} W.\ Barth, K.\ Hulek, C.\ Peters, A.\ Van de Ven, Antonius, \textit{Compact complex surfaces}.\ Second edition.\ EMG 4.\ Springer-Verlag, Berlin, 2004.\
    
    %\bibitem[CH11]{CH11} T.\ Crespo, Z.\ Hajto, \textit{Algebraic Groups and Differential Galois Theory}.\ GSM 122.\ AMS, Providence, 2011.\

    \bibitem[Chm21]{Chm21} T.\ Chmiel, \textit{Computing period integrals of rigid double octic Calabi--Yau threefolds with Picard--Fuchs operator}, J.\ Pure Appl.\ Algebra, vol.\ \textbf{225}, 2021
    
    \bibitem[CvS19]{CvS19} S.\ Cynk, D.\ van Straten, \textit{Periods of rigid double octic Calabi–Yau threefolds}, Ann.\ Polon.\ Math.\ \textbf{123}, 2019, pp.\ 243-258.\

    \bibitem[Cle77]{Cle77} C.\ H.\ Clemens, \textit{Degeneration of Kähler manifolds}.\ Duke Math.\ J., Volume \textbf{44} (1977) no.\ 1, pp.\ 215-290

    \bibitem[Del71]{Del71} P.\ Deligne, \textit{Formes modulaires et repr\'esentations $\ell$-adiques}, S\'eminaire N.\ Bourbaki (1971) \'exp.\ no \textbf{355}, pp.\ 139-172

    \bibitem[FK890]{FK890} R.\ Fricke and F.\ Klein, \textit{Vorlesungen \"uber die Theorie der elliptischen Modulfunctionen}, Teubner, Leipzig, 1890.\

    \bibitem[Grf68]{Grf68} P.\ Griffiths, \textit{On the periods of integrals on algebraic manifolds}, Rice Univ.\ Studies \textbf{54} (1968), 21-38.\ 

    \bibitem[Grf70]{Grf70} P.\ Griffiths, \textit{Periods of integrals on algebraic manifolds: Summary of main results and discussion of open problems}, Bull.\ Amer.\ Math.\ Soc.\ \textbf{76} (1970), 228–296
    
    %\bibitem[GH94]{GH94} P.\ Griffiths, J.\ Harris, \textit{Principles of Algebraic Geometry}, Wiley Interscience, New York, 1994.\

    \bibitem[Har77]{Har77} R.\ Hartshorne, \textit{Algebraic Geometry} (Springer, 1977).\
    
    %\bibitem[Har10]{Har10} R.\ Hartshorne, \textit{Deformation Theory} (Springer, 2010).\
    
    \bibitem[Her91]{Her91} S.\ Herfurtner, \textit{Elliptic surfaces with four singular fibres}, Mathematische Annalen, vol.\ \textbf{291} (1991), 319–342

    \bibitem[Hof13]{Hof13} J.\ Hofmann, \textit{Monodromy calculations for some differential equations.} Thesis, Johannes Gutenberg-Universität Mainz, Mainz, 2013.\
    
    \bibitem[Hus04]{Hus04} D.\ Husem\"uller, \textit{Elliptic Curves}, 2nd ed.\ (Springer, 2004).\

    %\bibitem[Huy05]{Huy05} D.\ Huybrechts, \textit{Complex geometry.\ An introduction}.\ Universitext.\ Springer-Verlag, Berlin, 2005.\

    \bibitem[Inc27]{Inc27} E.\ L.\ Ince, \textit{Ordinary Differential Equations}, Longmans, London, 1927.\

    \bibitem[KO68]{KO68} N.\ Katz, T.\ Oda, \textit{ On the differentiation of De Rham cohomology classes with respect to parameters}, J.\ Math.\ Kyoto Univ.\ \textbf{8-2} (1968), 199-213

    \bibitem[LLL82]{LLL82} A.K.\ Lenstra, H.W.\ Lenstra, L.\ Lovasz, \textit{Factoring Polynomials with Rational Coefficients}, Math. Ann., Vol. \textbf{261} (1982), 515-534

    \bibitem[LMFDB]{LMFDB} The LMFDB Collaboration, \textit{The L-functions and Modular Forms Database}, \url{http://www.lmfdb.org}, 2021, [Online; accessed 15 January 2021]

    \bibitem[LPV24]{LPV24} P.\ Lairez, E.\ Pichon-Pharabod, P.\ Vanhove, \textit{Effective homology and periods of complex projective hypersurfaces}, Math.\ Comp.\ \textbf{93} (2024), no.\ 350, 2985-3025

    \bibitem[Maple]{Maple} Maple 2022, 2023. Maplesoft, a division of Waterloo Maple Inc., Waterloo, Ontario.

    \bibitem[Mey05]{Mey05} C.\ Meyer, \textit{Modular Calabi--Yau Threefolds}, Fields Institute Monographs Vol.\ \textbf{22} (AMS 2005)

    \bibitem[Mez16]{Mez16} M.\ Mezzarobba, \textit{Rigorous Multiple-Precision Evaluation of D-Finite Functions in SageMath}. Extended abstract for a talk presented at the 5th International Congress on Mathematical Software (ICMS 2016), Berlin, Germany, available at: \url{https://hal.science/hal-01342769}

    \bibitem[Mil21]{Mil21} L.\ Milla, \textit{A detailed proof of the Chudnovsky formula with means of basic complex analysis}, preprint (2021) at arXiv:1809.00533v6

    \bibitem[Mir89]{Mir89} R.\ Miranda, \textit{The basic theory of elliptic surfaces}.\ Dottorato di Ricerca in Matematica.\ ETS Editrice, Pisa, 1989

    \bibitem[PP25]{PP25} E.\ Pichon-Pharabod, \textit{Periods of fibre products of elliptic surfaces and the Gamma conjecture}, preprint (2025) at arXiv:2505.07685v1

    \bibitem[PS08]{PS08}  C.A.M.\ Peters, J.H.M.\ Steenbrink, \textit{Mixed Hodge Structures}, A Series of Mod.\ Sur.\ in Math.\ \textbf{52} (Springer, 2008)

    \bibitem[RS20]{RS20} H.\ Ruddat, B.\ Siebert, \textit{Period integrals from wall structures via tropical cycles, canonical coordinates in mirror symmetry and analyticity of toric degenerations}.\ Publ.\ Math.\ IHES.\ \textbf{132} (2020), 1–82

    \bibitem[Sch88]{Sch88} C.\ Schoen, \textit{On fiber products of rational elliptic surfaces with section}, Mathematische Zeitschrift, vol.\ \textbf{197} (1988), 177–199

    \bibitem[Srt19]{Srt19} E.\ C.\ Sert\"oz, \textit{Computing periods of hypersurfaces}, Math.\ Comp.\ \textbf{88} (2019), no.\ 320, 2987-3022

    \bibitem[Sho81]{Sho81} V.\ V.\ Shokurov, \textit{The study of the homology of Kuga varieties}, Math.\ USSR-Izv.\, \textbf{16}:2 (1981), 399–418

    \bibitem[Stt04]{Stt04} M.\ Schütt, \textit{New examples of modular rigid Calabi--Yau threefolds}, Collect.\ Math.\ \textbf{55}.2 (2004), 219-228

    \bibitem[Voi02]{Voi02} C.\ Voisin (translated by L.\ Schneps), \textit{Hodge Theory and Complex Algebraic Geometry I}, Cambridge Stud.\ in Adv.\ Math.\ \textbf{76} (2002)
    
\end{thebibliography}
\end{document}